\documentclass[reqno]{amsart}
\usepackage[all]{xy}
\usepackage{color}
\usepackage{hyperref}
\usepackage{enumerate}
\usepackage{amsmath}
\usepackage{amsthm}
\usepackage{amssymb}
\usepackage{amscd}
\usepackage{graphicx}
\usepackage{epsfig}
\usepackage[english]{babel}
\usepackage[stable]{footmisc}

\let\oldmarginpar\marginpar
\renewcommand\marginpar[1]
{\oldmarginpar{\tiny\bf \begin{flushleft} #1 \end{flushleft}}}

  \newcommand{\rmap}{\longrightarrow}

\renewcommand{\d}{\mathrm d}               
\newcommand{\Lie}{\boldsymbol{\pounds}}    
\newcommand{\X}{\ensuremath{\mathfrak{X}}} 
\newcommand{\red}{{\mathrm{red}}} 
\renewcommand{\top}{{\mathrm{top}}} 
\newcommand{\can}{{\mathrm{can}}} 
\newcommand{\lin}{{\mathrm{lin}}} 

   \renewcommand{\a}{\alpha}
   \renewcommand{\b}{\beta}

  \newcommand{\w}{\omega}
  
   \newcommand{\R}{\mathbb{R}}
   
      \renewcommand{\S}{\mathbb{S}}
   \newcommand{\T}{\mathbb{T}}
  \newcommand{\Z}{\mathbb{Z}}

     \newcommand{\cC}{\mathcal{C}}
   \newcommand{\cD}{\mathcal{D}}
      
   \newcommand{\cF}{\mathcal{F}}
   \newcommand{\cG}{\mathcal{G}}

    \newcommand{\cH}{\mathcal{H}}

  \newcommand{\cN}{\mathcal{N}}

\newcommand{\G}{\mathcal{G}}            
\renewcommand{\O}{\mathcal{O}}             
\DeclareMathOperator{\Mon}{Mon}         
\renewcommand{\gg}{\mathfrak{g}}        

\newcommand{\tto}{\rightrightarrows}    

\DeclareMathOperator{\Ker}{Ker}           
\DeclareMathOperator{\im}{Im}           
\renewcommand{\Im}{\im}
\DeclareMathOperator{\Hol}{Hol}         

\DeclareMathOperator{\pr}{pr}      

\newcommand{\action}{\curvearrowright} 

\numberwithin{equation}{section}
\allowdisplaybreaks

\newtheorem{theorem}{Theorem}[section]
\newtheorem{lemma}[theorem]{Lemma}
\newtheorem{proposition}[theorem]{Proposition}
\newtheorem{corollary}[theorem]{Corollary}

\theoremstyle{definition}
\newtheorem{definition}[theorem]{Definition}
\newtheorem{example}[theorem]{Example}

\newtheorem{remark}[theorem]{Remark}

\begin{document}
\title{Poisson Manifolds of Compact Types (PMCT 1)}

\author{Marius Crainic}
\address{Depart. of Math., Utrecht University, 3508 TA Utrecht,
The Netherlands}
\email{m.crainic@uu.nl}

\author{Rui Loja Fernandes}
\address{Department of Mathematics, University of Illinois at Urbana-Champaign, 1409 W. Green Street, Urbana, IL 61801 USA}
\email{ruiloja@illinois.edu}

\author{David Mart\'inez Torres}
\address{Departamento de Matem\'atica, PUC-Rio, R. Mq. S. Vicente 225, Rio de Janeiro 22451-900, Brazil}
\email{dfmtorres@gmail.com}

\thanks{MC and DMT were partially supported by the ERC Starting Grant no. 279729.
RLF was partially supported by NSF grants DMS 13-08472 and DMS 14-05671, FCT/Portugal and the \emph{Ci\^encias Sem Fronteiras} program.}

\begin{abstract}
This is the first in a series of papers dedicated to the study of Poisson manifolds of compact types (PMCTs). This notion encompasses 
several classes of Poisson manifolds defined via properties of their symplectic integrations. In this first paper we establish some fundamental
properties and constructions of PMCTs. For instance, we show that their Poisson cohomology behaves very much like the
de Rham cohomology of a compact manifold (Hodge decomposition, non-degenerate Poincar\'e duality pairing, etc.)
and the Moser trick can be adapted to PMCTs. More important, we find unexpected connections between PMCTs and Symplectic Topology: PMCTs
are related with the theory of Lagrangian fibrations and we exhibit a construction of a nontrivial 
PMCT related to a classical question on the topology of the orbits of a free symplectic circle action. 
In subsequent papers, we will establish deep connections between PMCTs and Integral Affine Geometry,
Hamiltonian $G$-spaces, Foliation Theory, orbifolds, Lie Theory and symplectic gerbes.
\end{abstract}

\maketitle

\tableofcontents


\section{Introduction: a user guide to PMCTs}\label{sec:introduction}

A Poisson structure on a manifold is a far reaching generalization of the notion of a symplectic structure where one can
still define Hamiltonian dynamics. Roughly speaking, a Poisson structure is a (possibly) singular foliation
of the manifold by symplectic leaves. It is well known that the existence of a symplectic structure
on a \emph{closed} manifold is a delicate issue. On the other hand, general Poisson structures are 
extremely flexible objects: one always has the zero Poisson structure and, even locally around a point, 
different Poisson structures can have a very distinct behavior, in marked contrast with the symplectic case 
where Darboux's Theorem shows that there are no local invariants. 

There is little that can be said about a general Poisson structure, but upon imposing some restrictions, one obtains classes 
of Poisson structures with  rich geometry. For example, one can look at Poisson structures that are symplectic 
outside a codimension 1 submanifold, which include the b-symplectic (or log-symplectic) manifolds (\cite{log-sympl, b-sympl}). 
Or one can look at regular Poisson structures, i.e., symplectic foliations, for which the existence problem is wide open, even in the case of spheres of dimension greater than 5 (\cite{Mitsumatsu}).

In this series of papers we initiate the study of new classes of Poisson manifolds, that we call generically ``of compact type''. 
The name is inspired by the analogy with Lie algebras and Lie groups. One can proceed very naively and ask for the Lie algebra of 
smooth functions (endowed with the Poisson bracket) to be of compact type; the answer, of course, is almost never, since the smooth 
functions form an infinite dimensional vector space if the manifold has positive dimension. However, this question can be refined by passing to  the cotangent bundle of the Poisson manifold, which is a \emph{Lie algebroid} rather than a Lie algebra. One may then ask when this Lie algebroid integrates to a compact-like \emph{symplectic Lie groupoid}. Here, we use the expression ``compact-like'' because for a groupoid there are several possible variations of compactness. One may ask for:
\begin{itemize}
\item a compact symplectic Lie groupoid, i.e., a symplectic Lie groupoid whose space of arrows is compact;
\item a source-proper symplectic Lie groupoid, i.e., a symplectic Lie groupoid whose source map (and hence also target map) is proper;
\item a proper symplectic groupoid, i.e., a symplectic Lie groupoid whose anchor map is proper.
\end{itemize}
Moreover, just like the case of Lie algebras and Lie groups, one may ask for the cotangent Lie algebroid of the Poisson
manifold to integrate to \emph{some} Lie groupoid or rather to \emph{a source 1-connected} Lie groupoid, having one of these
properties. For example, a source-proper Poisson manifold is one whose Lie algebroid integrates to a source-proper symplectic Lie groupoid, 
while a strong source-proper
Poisson manifold is is one whose Lie algebroid integrates to a source 1-connected and source-proper symplectic Lie groupoid. 
Hence, there are 6 different, but related, classes of PCMTs and, as we will see, all of them are worth studying. 

Our study of PMCTs will exhibit many properties which show that these are very rigid objects, much like compact symplectic manifolds,
which they actually generalize. Here is a (very incomplete) list of properties of PMCTs:
\begin{enumerate}[(i)]
\item Poisson cohomology admits a Hodge decomposition and Poincar\'e duality holds;
\item there are natural operations such as fusion product, Hamiltonian quotients, gauge equivalence, etc., which preserve the PMCT nature;
\item leaves are embedded submanifolds and have finite holonomy;
\item there exist local linear models around leaves;
\item the leaf spaces are integral affine orbifolds;
\item the leafwise symplectic forms vary linearly in cohomology;
\item the symplectic volumes of the leaves define a piecewise polynomial function, relative to the integral affine structure on the orbit space;
\item the integral affine structure yields a canonical Hamiltonian invariant measure $\mu_{\textrm{Aff}}$, for which a Weyl-type integration formula holds;
\item in the $s$-proper case, there is a second canonical Hamiltonian invariant measure: the Duistermaat-Heckman measure $\mu_{\textrm{DH}}$, which is related to the affine measure $\mu_{\textrm{Aff}}$ via a polynomial formula.   
\end{enumerate}
All these properties distinguish PMCTs from general Poisson manifolds, placing them in a prominent position in Poisson Geometry. They are a manifestation of the deep connections of the theory of PMCTs with other subjects.

For example, the connection between PMCTs and Integral Affine Geometry, which will be explored fully only in the second and third papers in this series \cite{CFMb,CFMc}, turns out to be the key to properties (vi) and (vii), which are a generalization to PMCTs of the Duistermaat and Heckman result on the variation of cohomology classes of Hamiltonian reduced spaces \cite{DuHe}. Another example, also to be explored in \cite{CFMb,CFMc}, is the deep relation with the theory of compact Lie groups and Lie algebras, since fundamental results, such as the Weyl's integration formula or Weyl's covering lemma \cite{DK}, happen to be specific instances of geometric constructions for PMCTs. Likewise, we will see in a later section that the theory of quasi-Hamiltonian $G$-spaces leads to natural examples of PMCTs. However, in \cite{CFMb} we will see that not all examples of PMCTs arise in this way, and that the obstruction can be expressed via the non-triviality of a certain class associated with a \emph{symplectic gerbe}, providing yet another deep connection. There many other connections with Symplectic Geometry: for example, we will see later how PMCTs are related to the study of Lagrangian fibrations with compact connected fibers and, more generally, with  isotropic fibrations (see \cite{CFMb}); or how free quasi-Hamiltonian $\S^1$-actions on a compact symplectic manifold, whose orbits are contractible, yield PMCTs of \emph{strong type}. The existence of such actions was a longstanding problem in Symplectic Topology which was finally solved affirmatively by Kotschick in \cite{Ko}.

This first paper in the series is organized as follows. We start by providing in Section \ref{sec:def} some background
on Poisson structures, mainly with the aim of fixing notation. In Section \ref{sec:compactness} we formally define all 
6 classes of PMCTs that we study and we point out that these notions extend naturally to Dirac structures with a background. This degree of generality is needed even if one is interested only in PMCTs, and it also allows to include important examples of Dirac structures.
 In Section \ref{sec:examples} we give many basic examples of 
PMCTs which already exhibit some of their remarkable properties. Section \ref{sec:constructions} focuses on constructions
which allow to produce new examples of PMCTs out of known examples. In Section \ref{Poisson-homotopy-groups} we 
characterize Poisson manifolds of strong compact type in terms of the symplectic foliation and the variation of 
the symplectic forms on the leaves (more precisely, the monodromy groups). Section \ref{sec:basic-properties} 
discusses some basic Poisson-topological properties of the symplectic foliation of PMCTs as well as their Poisson cohomology,
the Hodge decomposition, and Poincar\'e duality. The last section of the paper shows that there is a local linear model for PMCTs
around leaves, which can be seen as a version of Moser's stability at the groupoid level.

Finally, we would like to point out that the work of Nguyen Tien Zung \cite{Zu} on proper symplectic groupoids should be considered as a precursor of the theory of PMCTs. However, Zung focus his attention on the symplectic groupoid, instead of the underlying Poisson manifold.
\newpage

\section{Background on Poisson Geometry}\label{sec:def}

\subsection{Poisson manifolds}\label{ssec:Poiss}

Let $(M,\{\cdot,\cdot\})$ be a Poisson manifold. Often we will specify it by giving a bivector $\pi\in \X^2(M)$, closed 
under the Schouten bracket: $[\pi,\pi]=0$. The two equivalent formulations are related to each other
via the formula
\[\pi(\d f,\d g)=\{f,g\}.\]
In our conventions, the Hamiltonian vector field of a function $h\in C^\infty(M)$ is the unique vector field $X_h\in\X(M)$ satisfying
\[ X_h(f)=\{h,f\},\quad \forall f\in C^\infty(M). \]
Therefore, $X_h=\pi^\sharp(\d f)$, where $\pi^\sharp:T^*M\to TM$ is contraction with $\pi$. 

The {\bf symplectic foliation} of $(M, \pi)$ is defined as
\[  \cF_{\pi}:= \pi^{\sharp}(T^*M)\subset TM, \]
and, hence, it is spanned at each point by Hamiltonian vector fields.
Of course, $\cF_{\pi}$ may be singular (i.e., of non-constant rank).
However, it is still completely integrable: through each point $x$ in $M$ there exists a maximal connected (immersed) 
submanifold $S$
with the property that $TS= \cF_{\pi}$ at all points of $S$. 
Each $S$ carries a canonical symplectic structure $\omega_{S}$ characterized by
\[ \omega_{S}(X_f, X_g)= \{f, g\}.\]
Accordingly, one talks about the symplectic leaves $S$ of $(M, \pi)$.

Roughly speaking, a Poisson structure can be thought of as ``a partition of $M$ into symplectic leaves'' fitting together smoothly. 
However, it is misleading to think of a Poisson structure as just the ``tangential data'' present in the symplectic
foliation. Crucial information of the Poisson structure is also encoded in the direction transverse to the leaves, and the understanding of the Poisson structure
requires a precise description of how the tangential and transverse data interact. Moreover, some of the most interesting behavior can happen
precisely at the singular leaves, i.e., the ones where the dimension of $\cF_{\pi}$ is not locally constant.

At the transverse level, we first have
the isotropy Lie algebras: for any $x\in M$ the {\bf isotropy Lie algebra} at $x$ is
\[ \gg_x(M, \pi):= \nu_{x}^{*}(S)=(T_xS)^o=(T_xM/T_xS)^* \]
with the Lie algebra structure described as follows: if $u=\d_xf, v=\d_xg$, 
for some smooth functions $f$ and $g$, then $[u, v]=\d_x\{f, g\}$.  We often write $\gg_x$ instead of 
$\gg_x(M, \pi)$, when it is clear the Poisson manifold in question. When $x$ varies in a leaf $S$, the isotropy Lie algebras
fit into a Lie algebra bundle $\gg_S\to S$ and we have a short exact sequence of vector bundles:
\begin{equation}\label{eq:liealgebrasbundlePoisson} \xymatrix{0\ar[r]& \gg_S\ar[r]& T^*_SM\ar[r]^{\pi^\sharp} & TS\ar[r] &0}.
\end{equation}

We say that $x\in M$ is {\bf regular point} for $(M, \pi)$ if there exists a neighborhood $V$ of $x$ where $\mathcal{F}_{\pi}|_U$ 
has constant dimension. Otherwise,
$x$ is called a {\bf singular point}. Note that at a regular point $x$ the isotropy Lie algebra $\gg_x$ is abelian. The subset of regular points
\[
M^{\textrm{reg}}= \{x\in M: x\ \textrm{is\ a\ regular\ point} \},\quad \mathcal{F}_{\pi}^{\textrm{reg}}= \mathcal{F}_{\pi}|_{M^{\textrm{reg}}},
\]
is an open dense subset of $M$ 
which is saturated with respect to the symplectic foliation, and $\mathcal{F}_{\pi}^{\textrm{reg}}$ is a regular foliation. 
Note, however, that the rank of $\mathcal{F}_{\pi}^{\textrm{reg}}$ may not be constant, since $M^{\textrm{reg}}$ may be disconnected.

Similar to the case of foliations, there is a the notion of {\bf linear holonomy} for Poisson manifolds. If $S$ is any 
leaf of $\cF_{\pi}$ a choice of splitting $\sigma:TS\to T^*_SM$
of (\ref{eq:liealgebrasbundlePoisson}) defines a connection on the bundle $\gg_S$ by setting:
\[ \nabla_X\d f:=-\Lie_{X_f}\sigma(X). \]
This connection restricts to a connection on the bundle of centers $Z(\gg_S)$ and the restriction
is flat and independent of the choice of splitting.  Parallel transport with respect to
this flat connection defines the linear holonomy representation:
\[
\rho: \pi_1(S,x)\to GL(Z(\gg_x)).
\]
Its image is a subgroup of $GL(Z(\gg_x))$ called the linear holonomy group of $(M, \pi)$ at $x$, and it is denoted by $\mathrm{Hol}_{x}(M, \pi)$. When 
$x$ is regular, then $Z(\gg_x)= \gg_x= \nu_{x}^{*}(S)$, and $\rho$ is just the dual of the usual linear holonomy 
at $x$ of the regular foliation $\mathcal{F}_{\pi}^{\textrm{reg}}$ 

Linear holonomy does not account for the symplectic structure of a leaf, but just for its topology. The {\bf monodromy
groups} are another important local invariants which incorporate symplectic information of the leaves. For PMCTs we will see that
these give rise to integral affine structures \cite{CFMb,CFMc}. The monodromy group at $x$ is an additive subgroup
\[  \cN_x \subset Z(\gg_x)\subset \gg_{x}= \nu_{x}^{*}(S) ,\]
defined as the image of a boundary map
\begin{equation}\label{monodromy-map-eq} 
\partial_x: \pi_2(S,x)\rmap Z(\gg_x)\subset \gg_{x}= \nu_{x}^{*}(S).
\end{equation}
This boundary map will be defined later (see Section \ref{ssec:liealg} and  Section \ref{Poisson-homotopy-groups}).
Intuitively, monodromy groups encode the variation of the symplectic areas 
of leafwise spheres (i.e., spheres that stay inside one leaf) along transverse directions,
and this interpretation can be
made precise at regular points  \cite{CF2}.

%

\begin{remark} The monodromy map (\ref{monodromy-map-eq}) is invariant under the action of
$\pi_1(S,x)$, hence the monodromy group $\cN_x$ is a subspace of $Z(\gg_x)$ that is invariant under the linear holonomy action.
\end{remark}

\subsection{The Lie algebroid point of view}\label{ssec:liealg} The cotangent bundle  $T^*M$ of a Poisson manifold carries a Lie algebroid structure. 
The  Lie algebroid point of view makes precise the relation between Poisson Geometry and Lie theory, and many of the constructions more transparent.

Recall that a {\bf Lie algebroid} over a manifold $M$ consists of a vector bundle $A$ over $M$, a Lie bracket $[\cdot, \cdot]$ on the space $\Gamma(A)$ of sections of $A$, and a 
vector bundle map $\rho: A\to TM$, called the anchor of $A$, satisfying the Leibniz-type identity: 
\[ [\alpha, f\beta]= f [\alpha, \beta]+ \Lie_{\rho(\alpha)}(f) \beta, \quad \forall\ \alpha, \beta\in \Gamma(A), f\in C^{\infty}(M).\]
One works with $A$ much in the same way as one works with the tangent bundle of a manifold, which is in fact an example of a Lie algebroid with the identity as anchor map. At the other extreme, when the base $M$ is a point, the definition reduces to the definition of a Lie algebra. Also, any foliation $\cF$ can be seen as a Lie algebroid with injective anchor map. 

We are mostly interested in the {\bf cotangent Lie algebroid} $T^*M$ of a Poisson manifold $(M, \pi)$ \cite{We1}: the anchor is $\pi^\sharp:T^*M\to TM$ and
the Lie bracket on sections of $T^*M$, i.e., on 1-forms, is given by:
\[ [\alpha,\beta]_\pi=\Lie_{\pi^\sharp(\alpha)}\beta-\Lie_{\pi^\sharp(\beta)}\alpha-\d\pi(\alpha,\beta). \]
In particular, on exact forms it is given by:
\[ [\d f, \d g ]_\pi=\d \{f,g\}. \]
Our preliminary discussion on Poisson Geometry in Section \ref{ssec:Poiss} can be recast using cotangent Lie algebroids. Moreover,
many of the properties introduced 
there  are common to all Lie algebroids, the main exception being the symplectic structures on the leaves. To start with, for any Lie algebroid $A\to M$
one can talk about
\begin{itemize}
\item the \emph{leaves} of $A$: these are the maximal connected immersed submanifolds $S$ of $M$ with the property that $T_yS= \rho(A_y)$ for all $y\in S$,
\item the \emph{isotropy Lie algebra} $\gg_x$ of $A$ at $x\in M$: this is just the kernel of the anchor map at $x$, 
with the Lie algebra structure induced from the Lie algebroid bracket.
\end{itemize}
Here are some further illustrations of the usefulness of the Lie algebroid point of view: cohomology, paths, and linear holonomy. 

\subsubsection{Cohomology} Thinking of Lie algebroids as replacements of tangent bundles of manifolds,
all the constructions on a manifold that depend only on the tangent bundle and the Lie derivative along vector fields, 
extend to the setting of Lie algebroids and, hence, can be applied to various particular classes
(e.g., to the cotangent Lie algebroid of a Poisson manifold). A simple instance of this principle is given by de Rham cohomology. 
For any Lie algebroid $A$ over a manifold $M$ one obtains the $A$-de Rham complex
\[ \Omega^{\bullet}(A):= \Gamma(\wedge^{\bullet}A^*) ,\]
endowed with the $A$-de Rham  differential $\d_A:\Omega^{\bullet}(A)\to \Omega_{A}^{\bullet+1}(M)$ given by the standard Koszul-type formula.
The resulting cohomology, referred to as the {\bf Lie algebroid cohomology} of $A$, and denoted by $H^{\bullet}(A)$,
is a generalization of both de Rham cohomology 
(obtained when $A= TM$) and  Lie algebra cohomology (obtained when $M=\{*\}$). In the case 
of the cotangent Lie algebroid $A= T^*M$ of a Poisson manifold $(M, \pi)$, this becomes the so-called {\bf Poisson cohomology} of
$(M, \pi)$, denoted by $H^{\bullet}_{\pi}(M)$. The defining complex consists of multivector fields
\[  \Omega^{\bullet}(T^*M)=\Gamma(\wedge^{\bullet}TM)= \X^{\bullet}(M),\]
and the differential $\d_{\pi}:\X^{\bullet}(M)\to\X^{\bullet+1}(M)$, coincides with taking the Schouten bracket with $\pi$:
\[ \d_{\pi}= [\pi, \cdot ].\]

\subsubsection{Paths} Inspired by the case of foliations and leafwise paths,
one can talk about leafwise paths for any Lie algebroid $A\to M$: a pair of paths, $\gamma:I\to M$ and $a:I\to A$,
such that $a(t)\in A_{\gamma(t)}$ and
\[ \rho(a(t))= \frac{\d}{\d t} \gamma(t) .\]
Since $\gamma$ is determined by $a$, one simply says that $a$ is an {\bf A-path}. This notion allows to describe the leaves of $A$ set-theoretically: 
two points are in the same leaf iff there exists an $A$-path with base path joining them.

The notion of $A$-path comes with an appropriate notion of ``$A$-homotopy", 
and the resulting homotopy classes of $A$-paths form a groupoid analogous to the homotopy groupoid of a manifold \cite{CF1}. 
The simplest non-trivial illustration of this construction is the case of Lie algebras: the resulting groupoid is actually a group, since the base is a point, 
and it is precisely the unique 1-connected Lie group $G(\gg)$ integrating the Lie algebra $\gg$ \cite{DK}.

\subsubsection{Linear holonomy} As another illustration of the use of the Lie algebroid point of view, 
let us describe, for any Lie algebroid $A\to M$ and $x\in M$, the linear holonomy 
\begin{equation}\label{lin-hol-gpd-case} 
\rho: \pi_1(S, x)\to GL(Z(\gg_x)),
\end{equation} 
where $S$ is the leaf of $A$ through $x$ and $\gg_x$ is the isotropy Lie algebra at $x$. In short, 
the kernel of $\rho$ at points in $S$ defines a Lie algebra bundle $\gg_{S}\rightarrow S$ with fiber $\gg_x$ at $x\in S$;
the Lie bracket of $A$ induces a flat connection on $Z(\gg_S)$ and $\rho$ is just the associated parallel transport. 
In more detail, the restriction of $A$ to $S$ defines a Lie algebroid $A_{S}$ over $S$ with surjective anchor map,
fitting into a short exact sequence
\begin{equation}\label{eq:liealgebrasbundlealgebroid}
\xymatrix{0\ar[r] & \gg_S\ar[r] & A_{S}\ar[r]^{\rho_S} \ar[r] & TS \ar[r]& 0.}
\end{equation}
The Lie bracket on $A_{S}$ induces an operation 
\[ \Gamma(A_S)\times \Gamma(\gg_S)\to  \Gamma(\gg_S),\quad (\alpha, \beta)\mapsto \nabla_{\alpha}(\beta):= [\alpha, \beta]\]
which has the formal properties of flat connections. By adapting the notion of parallel transport, but using $A$-paths,
it follows that any $A_S$-path $(a, \gamma)$ induces a parallel transport map by Lie algebra homomorphisms, which, upon restricting to
centers, defines a map
\[ \mathrm{hol}_{\gamma}^{\lin}: Z(\gg_x)\to Z(\gg_y)\]
only depending on the homotopy class of the base path $\gamma$  connecting $x$ to $y$. This defines the {\bf linear holonomy of $A$}. It coincides with parallel transport for the (ordinary) flat connection on the bundle $Z(\gg_S)\to S$ defined by:
\[ \nabla_X\a:=[\sigma(X),\a], \]
where $\sigma:TS\to A_S$ is any splitting of (\ref{eq:liealgebrasbundlealgebroid}).
%

\subsection{Symplectic groupoids}
Lie algebroids are the infinitesimal counterparts of Lie groupoids. Let us recall some  general notations and basic properties of Lie groupoids (see, e.g., \cite{CF4, MK, MM} for more details). We will write $\G\tto M$ to indicate that $\G$ is a {\bf Lie groupoid} over $M$, so $\G$ and $M$ are the 
manifolds of arrows and objects, respectively. The source and target maps
are denoted by $s$ and $t$, respectively. The unit at a point $x\in M$ is denoted by $1_x$. For $x\in M$ one has the 
$\G$-orbit $\O$ through $x$ (all points in $M$ connected to $x$ by some arrow in $\G$), the isotropy group 
$\G_x= s^{-1}(x)\cap t^{-1}(x)$
at $x$,
and the source-fiber (s-fiber for short) $s^{-1}(x)$. In general, $\O$ through $x$ is an 
immersed  submanifold of $M$ and $t: s^{-1}(x)\to \O$ is a principal $\G_x$-bundle over $\O$ (with the right action defined by the multiplication in $\G$).

We will soon concentrate on groupoids $\cG$ that are \textbf{source-connected} (s-connected for short), 
in the sense that all its s-fibers are connected.
These are the analogues of connected Lie groups in classical Lie theory. The analogue of passing to the identity component of a Lie group
is taking the s-connected subgroupoid $\cG^0\subset \cG$ 
that is made out of all the unit connected components of the s-fibers. 

Similarly, a groupoid $\cG$ is called \textbf{source $1$-connected} if its s-fibers are $1$-connected;
the analogue of the universal cover of a (connected) Lie group is the source-universal cover $\widetilde{\cG}$ that is made from the 
universal covers of the (connected) s-fibers of $\cG$ (see \cite{MM} for details).

Recall that any Lie groupoid $\G$ has an associated Lie algebroid $A= \textrm{Lie}(\G)$. As a vector bundle, $A$ consists of
the tangent spaces of the fibers $s^{-1}(x)$ at the units $1_x$. The anchor is given by the differential of
the target map. The Lie bracket arises by identifying $\Gamma(A)$ with the space of right-invariant vector
fields on $\cG$ tangent to the s-fibers, and using the standard Lie bracket of vector fields on $\cG$. Note that passing 
to the s-connected subgroupoid $\cG^0$ or to the source-universal cover $\widetilde{\cG}$ does not change the Lie algebroid. 

A Lie algebroid $A$ is called {\bf integrable} if it comes from a Lie groupoid. Any Lie groupoid $\G$ with
the property that $A$ is isomorphic to $\textrm{Lie}(\G)$ is called an integration of $A$. Of course, this terminology
comes from the fact that not every algebroid is integrable. However, this failure is by now quite well understood \cite{CaF,CF1,CF2}. 
Namely, there are computable obstructions that fully characterize integrability, and:
\begin{itemize}
\item if $A$ is integrable, then there is a canonical integration $\G(A)$ uniquely characterized (up to isomorphism) by the condition that
it is source $1$-connected;
\item if $A$ is integrable, then $\cG(A)$ can be built from any integrating Lie groupoid $\cG$ by passing
first to  $\cG^0$ and then to the source-universal cover;
\item if $A$ is integrable, any s-connected integration $\G$ of $A$ arises as a quotient of $\G(A)$: 
it comes with a morphism of Lie groupoids $\G(A)\rmap \G$ which is a surjective local diffeomorphism;
\item $\G(A)$ is always defined as a topological groupoid and the integrability of $A$ is equivalent to the smoothness of $\G(A)$. 
\end{itemize}
As already mentioned, $\G(A)$ arises as the analogue of the homotopy groupoid of a manifold, but using $A$-paths:
\[ \G(A)= \frac{A-\textrm{paths}}{A-\textrm{homotopies}},\]
with source/target the maps obtained by taking the end points of the base path and composition induced by concatenation (see \cite{CF1,CF4} for details).

%

Returning to Poisson Geometry, it is clear that the ``global counterpart" of a Poisson manifold $(M, \pi)$ is the groupoid $\G(A)$ associated to  its cotangent Lie algebroid $A= T^*M$. It will be denoted by 
\[ \Sigma(M, \pi)\tto M\]
and called the \textbf{Weinstein groupoid} of $(M, \pi)$. 

In this paper we will only be interested in the integrable case, i.e., when $\Sigma(M, \pi)$ is a Lie groupoid. 
It then follows that $\Sigma(M, \pi)$ carries a canonical symplectic structure $\Omega_{\Sigma(M,\pi)}$ compatible
with the multiplication \cite{CaF,CF1},
i.e., it is a \textbf{symplectic groupoid}, and $\pi$ can be
recovered as the unique Poisson structure on $M$ with the property that $t$ becomes a Poisson map.

The Weinstein groupoid brings together the various pieces of Poisson Geometry already discussed (and reveals a few more),
into a global object inside which they interact. 
For instance, while the symplectic leaves of $(M, \pi)$ are the orbits of $\Sigma(M, \pi)$,
the isotropy Lie algebras $\gg_x(M, \pi)$ are the Lie algebras of the isotropy Lie groups $\Sigma_x(M, \pi)$. 
The two interact inside $\Sigma(M, \pi)$ through the principal $\Sigma_{x}(M, \pi)$-bundle
over $S$ given by the $s$-fiber at $x$. This principal bundle $t: s^{-1}(x)\to S$ plays the role of ``the Poisson homotopy cover''
of the symplectic leaf. Similarly:

\begin{definition} The {\bf Poisson homotopy group} of $(M, \pi)$ at $x\in M$ is the isotropy group $\Sigma_x(M, \pi)$ at $x$ 
of the Weinstein groupoid $\Sigma(M, \pi)$.
\end{definition}

In the integrable case, $\Sigma_x(M, \pi)$  is a Lie group integrating the isotropy
Lie algebra $\gg_x= \gg_x(M, \pi)$. In general, $\Sigma_x(M, \pi)$ is always defined as a topological group and
it is an interesting local invariant of the Poisson structure. 

The Poisson homotopy group brings us  back to the monodromy groups of $(M, \pi)$ introduced in our
initial discussion on Poisson Geometry. In the integrable case, 
a definition equivalent to the one using (\ref{mon-gp-gen}) is the following:


\begin{definition} \label{def-mon-gps-int}
Given an integrable Poisson manifold $(M, \pi)$, the monodromy group of $(M, \pi)$ at $x\in M$ is defined as
\[
\cN_x= \{u\in Z(\gg_x(M, \pi )): \textrm{exp}(u)= 1\},
\]
where $\textrm{exp}:\gg_x(M, \pi )\to \Sigma_x(M, \pi)$ is the exponential map of the Poisson homotopy group at $x$.
%
%
\end{definition}

Later, we will recall more details on these groups, including their relevance to integrability (see, in particular, Section \ref{Poisson-homotopy-groups}).

\begin{remark}[On Hausdorffness]
\label{rem:Hausdorff} 
In this paper we will only consider Hausdorff Lie groupoids which are s-connected.
The general notion of Lie groupoid allows the manifold of arrows to be non-Hausdorff, although one still requires
the s-fibers and the base to be Hausdorff manifolds. The reason is that even for some very simple examples of foliations 
(or, even, of bundles of Lie algebras) the resulting integrating groupoids may be non-Hausdorff.
Because of such examples, we will have to pay extra-attention to ensure that the space of arrows of our Lie groupoids is always Hausdorff 
and, likewise, that the s-fibers are connected.
\end{remark}

\section{Compactness types in Poisson Geometry}
\label{sec:compactness}

\subsection{Compactness types: proper, s-proper, compact}
\label{Compactness types: proper, s-proper, compact}
The ``compactness" of a Poisson manifold, as in the case of Lie algebras, should be defined via the corresponding global object,
i.e., the associated symplectic groupoid. But when should one declare a Lie groupoid $\cG$ to be ``of compact type"? 
The most obvious condition is to simply require the space of arrows $\cG$ to be 
\textbf{compact}\footnote{As a general convention, when we say that a manifold is compact we assume also that it is Hausdorff. Likewise,
the properness of a map includes the requirement that all the spaces
involved are Hausdorff (see Remark \ref{rem:Hausdorff}).}. On the other hand, recalling the appearance of the s-fibers in passing from 
Lie groups to Lie groupoids, another natural condition is to require the source map
$s: \cG\to M$ to be proper (i.e., compactness  
of the s-fibers). In this case we will say that $\cG$ is \textbf{s-proper} (a short form of source-proper). However,
although this may not be obvious, it turns out that the true generalization of compactness from the Lie group to the Lie groupoid setting is the condition that the groupoid anchor:
\[ (s, t): \cG\to M\times M,\]
is a proper map. In this case one says that $\cG$ is \textbf{proper}. 

\begin{example}\label{ex-act-gpd} For a Lie group, interpreted as a Lie groupoid over a point,
the three notions are equivalent.  However, if $G\ltimes M$ is the Lie groupoid associated to the action of
a Lie group $G$ on a manifold $M$, then $G\ltimes M$ is:
\begin{itemize}
\item  compact iff $G$ and $M$ are compact;
\item  s-proper iff $G$ is compact;
\item  proper iff the action of $G$ on $M$ is proper. 
\end{itemize}
It is well-known that most of the fundamental properties of group actions under the various
``compactness conditions" hold under the weakest one: properness (see, e.g., \cite{DK}).
\end{example}

Another reason why compactness of $\cG$ may not always be desirable is that this property is not invariant under Morita equivalence.
On the other hand, both s-properness and properness are Morita invariant, so these later properties 
are really about the transverse geometry of the characteristic foliation. The implications of this fact in Poisson Geometry will be explored
in the future papers in this series \cite{CFMb,CFMc}.


We will be interested in all the types of ``compactness'' mentioned above. Generically, we will refer to them as ``$\mathcal{C}$-type", with 
\begin{equation}
\label{types}
\mathcal{C}\in \{ \textrm{proper,\ s-proper,\ compact}\}.
\end{equation}

\begin{definition}
\label{def:pmct}
Let $\mathcal{C}$ be one of the three compactness types (\ref{types}). A Poisson manifold $(M, \pi)$ is said to be 
\begin{itemize}
\item \textbf{of $\mathcal{C}$-type} if there exists an s-connected symplectic integration of $(M,\pi)$
having the property $\mathcal{C}$;
\item \textbf{of strong $\mathcal{C}$-type} if the canonical integration \[(\Sigma(M),\Omega_\Sigma):=(\Sigma(M,\pi),\Omega_{\Sigma(M,\pi)})\]
is smooth and has the property $\mathcal{C}$.
\end{itemize}
\end{definition}
We will also refer generically to the Poisson manifolds covered by Definition \ref{def:pmct} as {\bf Poisson manifolds of compact types} (PMCTs for short).

\begin{remark}[Relationship between the various $\mathcal{C}$-types] \rm \
In general,
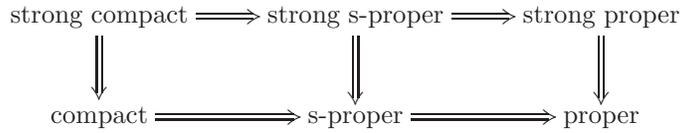
\[ \xymatrix{
\textrm{strong compact} \ar@{=>}[r] \ar@{=>}[d] & \textrm{strong\ s-proper} \ar@{=>}[r] \ar@{=>}[d] &  \textrm{strong\ proper} \ar@{=>}[d]\\
\textrm{compact} \ar@{=>}[r] &  \textrm{s-proper} \ar@{=>}[r] &  \textrm{proper}}\]
The difference between the s-proper and proper types lies in the compactness of the symplectic leaves.
When all the \emph{symplectic leaves are compact}, the previous diagram becomes:
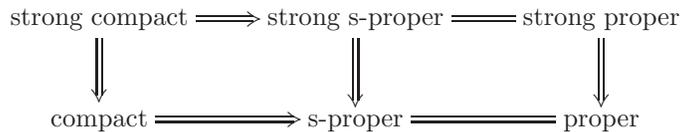
\[
\xymatrix{
\textrm{strong compact} \ar@{=>}[r] \ar@{=>}[d] & \textrm{strong\ s-proper} \ar@{=}[r] \ar@{=>}[d] &  \textrm{strong\ proper} \ar@{=>}[d]\\
\textrm{compact} \ar@{=>}[r] &  \textrm{s-proper} \ar@{=}[r] &  \textrm{proper}}
\]
Similarly, the difference between compact type and s-proper (or proper) type lies in the compactness of $M$. When \emph{$M$ is a compact 
manifold}, the diagram becomes:
\[
\xymatrix{
\textrm{strong compact} \ar@{=}[r] \ar@{=>}[d] & \textrm{strong\ s-proper} \ar@{=}[r] \ar@{=>}[d] &  \textrm{strong\ proper} \ar@{=>}[d]\\
\textrm{compact} \ar@{=}[r] &  \textrm{s-proper} \ar@{=}[r] &  \textrm{proper}}
\]
\end{remark}


\subsection{Variation: Dirac structures}
The notions of $\mathcal{C}$-types and strong $\mathcal{C}$-types can be adapted to various other structures 
associated with Lie algebroids. For instance, for Lie algebras (i.e., Lie algebroids over a point),
the resulting notion of ``compact type" coincides with the usual one, whereas Lie algebras of ``strong compact type" 
are those of compact type which are semi-simple. Another interesting example, more relevant to Poisson Geometry, 
is that of foliations, viewed as Lie algebroids with injective anchor map. We defer a detailed analysis of the 
``compactness'' conditions for foliations to \cite{CFMb}.

Here we consider the case of Dirac structures. While they provide a generalization of Poisson structures, 
the main reason we consider them is not the resulting greater generality but their relevance to the Poisson case. 
For instance, several constructions, such as gauge equivalence (Section \ref{sub-sub:Gauge Equivalences}) or 
normal forms \cite{CFMc} are natural in the realm of Dirac geometry. Even more, the ``desingularization" of PMCTs, 
to be discussed in \cite{CFMc}, will take us into the Dirac setting.

Let us first recall a few basic concepts on Dirac geometry. Given a manifold $M$,
its generalized tangent bundle $TM\oplus T^*M$  is endowed with both a fiberwise metric of split signature
\begin{equation}\label{eq:genin}
\langle X+\alpha,Y+\beta\rangle_+=\frac{1}{2}(\beta(X)+\alpha(Y)),\, X,Y\in\mathfrak{X}(M),\,\alpha,\beta\in \Omega^1(M),
 \end{equation}
and a skew symmetric bracket (which fails to satisfy the Jacobi identity):
\begin{equation}\label{eq:genbra}
 [X+\alpha,Y+\beta ]=X+Y+\mathcal{L}_X\beta-\mathcal{L}_Y\alpha-\frac{1}{2}(d(i_X\beta-i_Y\alpha)).
\end{equation}

A \textbf{Dirac structure} on $M$ is a subbundle $L\subset TM\oplus T^*M$ of the generalized tangent bundle,
which is both maximally isotropic with respect to the metric (\ref{eq:genin}) and involutive with respect to the 
generalized bracket (\ref{eq:genbra}). This automatically turns $L\rightarrow M$ into a
Lie algebroid with anchor map the restriction of the first projection $TM\oplus T^*M\rightarrow TM$. 

Given a Dirac structure $L$ its presympectic foliation is, by definition, the characteristic foliation of its Lie algebroid. Each
leaf of this foliation is endowed with a closed 2-form, which is obtained by restricting to $L$ the natural skew-symmetric form:
\[
\langle X+\alpha,Y+\beta\rangle_-=\frac{1}{2}(\beta(X)-\alpha(Y)),\, X,Y\in\mathfrak{X}(M),\,\alpha,\beta\in \Omega^1(M).
\]
The non-degeneracy condition that distinguishes Poisson structures among Dirac structures, is that the presymplectic 
forms induced on the leaves of the characteristic foliation must be symplectic. This happens iff the subbundle
$L$ defining the Dirac structure is the graph of a bivector $\pi$:
\[ L_\pi=\{ (\pi^\sharp(\alpha),\alpha):\alpha\in T^*M \}. \]
For a general Dirac structure $L$ one defines the {\bf Poisson support} of $L$:
\begin{equation}\label{Poisson-support}
\textrm{supp}(L)=\{ x\in M: \textrm{pr}_2= T^{*}_{x} M\}\subset M,
\end{equation}
which is the largest open in $M$ on which $L$ is given by a Poisson bivector $\pi_L$.

Any closed 2-form $\omega\in\Omega^2(M)$ with non-trivial kernel leads to an example of a Dirac structure which is not Poisson: 
\[ L_\omega=\{ (X, i_X\omega):X\in TM \}. \]
In this example the presymplectic foliation has the single leaf $M$ and 
its Poisson support consists precisely of the points where $\omega$ is non-degenerate. 
%
%

A Dirac structure $(M,L)$ is called integrable if the Lie algebroid $L\rightarrow M$ is integrable. 
The canonical integration with 1-connected fibers will still be denoted by $\Sigma(M, L)\tto M$ and called the Weinstein
groupoid of $(M, L)$. This  Lie groupoid also carries a closed 2-form, but which may be degenerate, albeit in a very controlled way. In general, 
a {\bf presymplectic groupoid} $(\G\tto M,\Omega)$ is a Lie groupoid $\G$ endowed with a multiplicative 2-form $\Omega$ such that:
\[ \Ker\Omega\cap\Ker \d s\cap \Ker\d t=\{0\}. \]
As in the symplectic case, $\dim \G=2\dim M$ and  the base manifold carries a canonical Dirac structure $L$.
We then say that $(\G,\Omega)$ is a presymplectic
groupoid integrating $(M, L)$.  We refer to \cite{BCWZ} for a detailed account.

It should now be clear how to define \textbf{Dirac structures of $\mathcal{C}$-type or strong
$\mathcal{C}$-type}, where $\mathcal{C}$ is any of the compactness types (\ref{types}): in Definition
\ref{def:pmct} one replaces ``symplectic'' by ``presymplectic" and $\Sigma(M, \pi)$ by $\Sigma(M,L)$. 

Recalling that the Poisson support (\ref{Poisson-support}) is an open saturated set, we obtain the following relationships:
\begin{proposition} For $\mathcal{C}\in \{\textrm{proper}, \textrm{s-proper}\}$, if $(M, L)$ is a 
Dirac manifold of (strong) $\mathcal{C}$-type, then its Poisson support $(\textrm{supp}(L),\pi_L)$ is a Poisson
manifold of (strong) $\mathcal{C}$-type. If $(M, L)$ is a Dirac manifold of (strong) compact type, then its 
Poisson support $(\textrm{supp}(L),\pi_L)$ is a Poisson manifold of (strong) s-proper type.
\end{proposition}

Finally, we recall that one can enlarge the class of Dirac structures by allowing for a {\bf background closed 3-form} $\phi\in\Omega^3(M)$ \cite{SW}. One defines a {\bf $\phi$-twisted Dirac structure} by replacing in the definition the bracket (\ref{eq:genbra}) by its $\phi$-twisted version:
\[ [X+\alpha,Y+\beta ]_{\phi}:= [X+\alpha,Y+\beta ]+ i_Xi_Y(\phi).\]
The discussion regarding integrability extends in a straightforward manner provided one uses $\phi$-twisted presymplectic
groupoids $(\G,\Omega)$, i.e., now the 2-form is not closed anymore, but satisfies the twisting condition:
\[ \d \Omega=t^*\phi-s^*\phi. \]
The definition of $\phi$-twisted Dirac structure of $\mathcal{C}$-type or strong $\mathcal{C}$-type is now obvious.
We shall use the abbreviation {\bf DMCTs} to refer to any of these (twisted) Dirac manifolds of compact types.

\section{First examples of PMCTs} 
\label{sec:examples}
In this section we will introduce several classes of examples of PMCTs. Most notably, 
these examples will uncover a fundamental connection between PMCTs and integral affine structures.

\subsection{(Pre)symplectic manifolds}
\label{sec:ex:symplectic}
Symplectic manifolds $(S,\omega)$ correspond to non-degenerate Poisson structures $\pi$. The symplectic foliation has only one leaf and
the anchor map $\pi^{\sharp}: T^*S\to TS$ is an isomorphism between cotangent and tangent Lie algebroids. 
Just like for foliations (in fact, as a particular case of them), the tangent Lie algebroid has two extreme integrations:
\begin{itemize}
\item the pair groupoid $S\times S\tto S$, with source/target map the second/first projection and  multiplication given by
$(x, y)\cdot (y, z)= (x, z)$.
\item the fundamental groupoid $\Mon(S)$, consisting of homotopy classes of paths in $S$ (relative to the endpoints). If $\widetilde{S}$
denotes the universal cover of $S$, this groupoid can also be described as the quotient of the pair groupoid of $\widetilde{S}$ modulo 
the diagonal action of $\pi_1(S)$:
\[ \Mon(S)\cong(\widetilde{S}\times \widetilde{S})/\pi_1(S). \] 
\end{itemize}
Any other s-connected integration $\cG$ lies in between these two, since there are surjective groupoid submersions:
\[ \xymatrix{\Mon(S)\ar[r] & \cG \ar[r] & S\times S}. \]
More precisely, integrations $\cG$ as above are in 1-1 correspondence with covering spaces $\hat{S}$ of $S$:
\[ \cG\cong(\hat{S}\times \hat{S})/\Gamma,\] 
where $\Gamma$ is the corresponding group of deck transformations. Moreover, all these groupoids are symplectic,  
with the multiplicative symplectic form 
\[ \Omega=t^*\omega-s^*\omega.\]
We deduce that $(M,\omega)$ is:
\begin{itemize}
\item  always of proper type;
\item of compact type iff it is s-proper iff $M$ is compact; 
\item of strong proper type iff the fundamental group of $M$ is finite;
\item of strong compact type iff it is strong s-proper iff $M$ is compact and has finite fundamental group.
\end{itemize}
The previous discussion applies more generally to all (twisted) Dirac structures $L_{\omega}$ associated with 
a 2-form $\omega\in \Omega^2(S)$ (i.e., presymplectic manifolds $(S, \omega)$). The conclusions are the same. 

\subsection{The zero Poisson structure and integral affine structures}
\label{sec:integral:affine}
The zero Poisson structure  $\pi\equiv 0$ on a manifold $M$ has symplectic leaves the points of $M$.
This does not appear to be a very interesting object, but we claim that even for this trivial
Poisson structure Definition \ref{def:pmct} has some interesting
nontrivial consequences! 

First of all, the Weinstein groupoid of $(M, 0)$ is 
\[ (\Sigma(M, 0),\Omega_{\Sigma(M,0)})=(T^*M,\omega_{\textrm{can}})\] 
viewed as a bundle of Abelian Lie groups (with addition on the fibers). Obviously, $(M, 0)$ is never of strong proper type. However,
it may be of proper type (hence, also s-proper type).  In fact, a choice of a proper symplectic integration is equivalent
to the choice of an integral affine structure on $M$. This is the first sign of the relevance of integral affine structures to the study of PMCTs. 
First, we recall:


\begin{definition}\label{def:ia-atlas} An {\bf integral affine structure} on a $q$-dimensional manifold $M$ is a choice of a maximal
atlas $(U_i,\phi_i)_{i\in I}$, such that each transition function 
\[ \phi_{j}\circ\phi^{-1}_i:\phi_i(U_i\cap U_j)\to \phi_j(U_i\cap U_j) \]
is (the restriction of) an integral affine transformation
\[
\mathbb{R}^q\to \mathbb{R}^q, \quad u\mapsto u+ A(x),\quad u\in \mathbb{R}^q,\quad A\in \mathrm{GL}_\Z(\R^q).
\]
\end{definition}

The coordinates $(U_i,\phi_i)$ in the special atlas defining an integral affine structure are called \emph{integral affine coordinates}. 

By a {\bf lattice} in a vector space $V$ we will mean a discrete subgroup of $(V, +)$ of maximal 
rank (often called in the literature a full rank lattice). By a (smooth)  lattice in a vector bundle $E\to M$ we mean a 
sub-bundle $ \Lambda\subset E$ whose fibers $\Lambda_x$ consist of lattices in the vector spaces $E_x$. In particular $\Lambda$ is a 
submanifold of $E$ transverse to the fibers and of the same dimension as $M$. Integral affine structures on $M$ are in 1-1 correspondence with \emph{Lagrangian lattices} in $T^*M$, i.e., lattices $\Lambda\subset T^*M$ such that the pullback of $\omega_{\textrm{can}}$ to $\Lambda$ is zero: $(U,x^1,\dots,x^q)$ is an integral affine coordinate chart if and only if for any $x\in U$ one has: 
\[ \Lambda_x=\Z\langle \d_x x^1,\dots,\d_x x^q\rangle. \]
 
\begin{proposition}\label{pro:lag-zero}
Proper symplectic integrations of $(M,0)$ are in 1-1 correspondence with integral affine structures on $M$.
\end{proposition}
\begin{proof}
Let $(\G,\Omega)$ be a proper integration of $(M,0)$. Then $\G$ is isomorphic to a quotient $T^*M/\Lambda$,
where $\Lambda$ is a (full rank) lattice due to the compactness of the source/target fibers of $\G$. Moreover, since the 
canonical symplectic form $\omega_{\textrm{can}}$ on $T^*M$ descends to $T^*M/\Lambda$, the lattice $\Lambda$ must be Lagrangian.
Conversely, any integral affine structure on $M$ defines a Lagrangian lattice $\Lambda\subset T^*M$, and hence provides an s-proper 
symplectic integration $(T^*M,\omega_{\textrm{can}})/\Lambda$ of $(M,0)$.
\end{proof}

Integral affine structures will be discussed at length in \cite{CFMb,CFMc}, where they will be used to prove that any PMCT admits an \emph{orbifold integral affine structure} on its leaf space.

\subsection{Symplectic fibrations over integral affine manifolds}

Symplectic fibrations $p:M\to B$ provide an interesting class of regular Poisson manifolds
with smooth leaf space. When the base is an integral affine manifold and, e.g. the fibers are 1-connected, $(M,\pi)$ is of proper type, as we will now explain.

A symplectic fibration can be characterized by the existence of a 2-form $\omega$ on $M$ extending the leafwise symplectic forms and satisfying:
\begin{equation}
\label{eq:symp:fibration}
i_X i_Y\d\omega=0,\quad \forall X,Y\in\ker\d p.
\end{equation}
When $M$ is compact, this condition is equivalent to local triviality of the fibration as a symplectic fiber bundle (see \cite[Chapter 1]{GLS}).
The canonical integration $\Sigma(M,\pi)$ is the action groupoid
\[ \Mon(\cF_\pi)\ltimes \nu^*(\cF_\pi) \tto M, \]
where $\Mon(\cF_\pi)$ acts on the conormal bundle $\nu^{*}(\cF_\pi)$ via the linear holonomy representation (i.e., by parallel transport
relative to the Bott connection). The symplectic form is of type:
\[ \Omega=\pr^*_{\Mon}\Omega_{\Mon}+\pr^*_{\nu^*}\omega_\can, \]
where $\omega_\can$ is the canonical 2-form on $\nu^*(\cF_\pi)$ and $\Omega_{\Mon}$ is a multiplicative 2-form on
the monodromy groupoid. When \eqref{eq:symp:fibration} is replaced by the stronger condition:
\begin{equation}
\label{eq:symp:fibration:stronger}
i_X\d\omega=0,\quad \forall X\in\ker\d p,
\end{equation}
then:
\[ \Omega_{\Mon}=t^*\omega-s^*\omega. \]
In the general case this form may fail to be closed- problem which fits in the usual theory of symplectic fibrations, but this time for the fibration $s:\Mon(\cF)\to M$
with the leafwise symplectic forms induced by $\Omega_{\Mon}$. To replace $\Omega_{\Mon}$ by a closed (multiplicative) 2-form one can make use of the standard theory 
of coupling forms that is available for fibrations with 1-connected fibers (see \cite[Chapter 6]{McDS}). In our context however there is a construction of the coupling form that is 
more canonical and which also allows us to treat other integrations of $T^*M$. The idea is to modify $\Omega_{\Mon}$ by a 2-form $B\in\Omega^2(\Mon(\cF_\pi))$ which is horizontal with respect to the s-fibers; such 2-forms are automatically multiplicative. If we identify the normal space to the s-fiber at $[\gamma]\in \Mon(\cF_\pi)$ with $\nu_{\gamma(0)}(\cF_\pi)$, then
 $B$ can be seen as a groupoid 1-cocycle with values in the representation $\wedge^2\nu^*(\cF_\pi)$: 
\[ B\in C^{1}_{\text{diff}}(\Mon(\cF_\pi),\wedge^2\nu^*(\cF_\pi))\subset \Omega^2(\Mon(\cF_\pi)).\] 
As such, $B$ is defined as the integration of the algebroid 1-cocycle:
\[ b\in\Omega^1(\cF_\pi,\wedge^2\nu^*(\cF_\pi)),\quad X\mapsto i_X\d\omega, \]
the integration exists because the s-fibers of $\Mon(\cF_\pi)$ are 1-connected. 

Note that $\Mon(\cF_\pi)$ might fail to be proper/Hausdorff. Replacing it by the holonomy groupoid $\Hol(\cF_\pi)=M\times_B M$  results in the action groupoid:
\[ (M\times_B M)\ltimes \nu^*(\cF_\pi) \tto M. \]
However, in general, this groupoid fails to be symplectic. The reason is that the cocycle $b\in\Omega^1(\cF_\pi,\wedge^2\nu^*(\cF_\pi))$ 
may not integrate to  $M\times_BM$. Since this groupoid is proper, $b$ integrates if and only if it is a coboundary. 
Hence the obstruction for this integration to be symplectic is the cohomology class:
\[ [b]\in H^1(\cF_\pi,\wedge^2\nu^*(\cF_\pi)). \]
This class vanishes, e.g., when the leaves of $\cF_\pi$ are 1-connected.

Assuming $[b]=0$ we obtain a symplectic  integration which is not yet proper. However, if $(B,\Lambda)$ is an integral affine manifold then we can further quotient by the lattice $p^*\Lambda$, resulting in the proper, symplectic integration: 
\[ (M\times_B M)\ltimes (\nu^*(\cF_\pi)/p^*\Lambda) \tto M. \]
Finally, notice that the condition $[b]=0$ amounts to the existence of an extension $\omega$ satisfying the 
stronger condition \eqref{eq:symp:fibration:stronger}. In turn, this means that the associated symplectic
connection has Hamiltonian curvature (see \cite[Chapter 6]{McDS}).

%

\subsection{Foliations and transverse integral affine structures}
\label{sec:integral:affine:foliation}
The same way any manifold carries the zero Poisson structure, any (regular) foliation carries a Dirac structure with zero presymplectic 
forms. More precisely, any foliation $\cF$ on $M$ gives rise to the Dirac structure:
\[ L_{\cF}:= \cF \oplus \nu^{*}(\cF).\]
Clearly, the presymplectic leaves are the leaves of $\cF$ endowed with the zero presymplectic forms. The source 1-connected integration of $L_\cF$ is obtained as above:
\[ \Sigma(M,L_\cF)=\Mon(\cF)\times_M \nu^{*}(\cF)\tto M, \]
with presymplectic form $\Omega_\Sigma=\pr^*_{\nu^*}\omega_\can$.

One can obtain s-connected proper presymplectic integrations of $(M, L_{\cF})$ by considering a Lagrangian lattice $\Lambda\subset \nu^{*}(\cF)\subset T^*M$. Such a lattice is invariant under the action of the holonomy groupoid. Any s-connected groupoid $\cG$ integrating $\cF$ sits in between the monodromy and the holonomy groupoid:
\[ \Mon(\cF)\to \cG\to \Hol(\cF). \]
Hence, if $\cG$ is proper we obtain the s-connected proper presymplectic integration:
\[ \cG\times_M \nu^*(\cF)/\Lambda\tto M. \]
In particular, when the leaves of $\cF$ are compact with finite fundamental group one can take any integration $\cG$ of $\cF$.


Just like a Lagrangian lattice $\Lambda\subset T^*M$ defines an integral affine structure on $M$, a Lagrangian lattice 
$\Lambda\subset \nu^{*}(\cF)\subset T^*M$ defines a {\bf transverse integral affine structure} on $\cF$.
It amounts to a special choice of a foliated atlas, and it can be thought of as an integral affine structure on the leaf space $M/\cF$. Transverse integral affine structures will be discussed at length in \cite{CFMb,CFMc}.


\subsection{Linear Poisson structures}
\label{sec:linear:Poisson}
Recall that a Poisson bracket on a  vector space $V$ is called linear if the the Poisson bracket of linear functions is again a linear function. 
This happens precisely when $V=\gg^*$ for some  Lie algebra $\gg$, and then the Poisson bracket is given by the Kostant-Kirillov-Souriau formula:
\[ \{f,g\}(\xi)=\langle \xi,[\d_\xi f,\d_\xi g]_\gg\rangle, \quad \forall f,g\in C^\infty(\gg^*),\ \xi\in\gg^*. \]

The s-connected symplectic groupoids integrating the linear Poisson structure
on $\gg^*$ are of the form $(T^*G, \omega_{\textrm{can}})$, where $G$ is any connected Lie group
integrating $\gg$. 
Here $T^*G$ is viewed as a groupoid over $\gg^*$ with source/target maps the right/left
trivializations:
\[ s(\a_g)=(\d R_g)^*\a_g,\qquad t(\a_g)=(\d L_g)^*\a_g. \]
Equivalently, $T^*G= G\ltimes \gg^*$ is the action groupoid associated to the coadjoint action of $G$ on $\gg^*$. 
It follows that:
\begin{itemize}
\item $\gg^*$ is proper (respectively, strong proper)  iff s-proper (respectively, strong s-proper) iff there exists some compact (respectively, compact 1-connected) Lie group integrating $\gg$, i.e.,
iff  $\gg$ is a Lie algebra of compact type (respectively, semisimple of compact type);
\item $\gg^*$ is never of compact type.
\end{itemize}

\subsection{Cartan-Dirac structures}
There is a Dirac structure of compact type associated to any Lie group $G$ whose Lie algebra $\gg$ admits an invariant 
scalar product $\langle \cdot,\cdot\rangle$. This class of Dirac structures is the analogue of the linear Poisson structures on $\gg^*$.
The main example to keep in mind is when $G$ is a compact Lie group. It is instructive to have a look at the various
basic results on compact Lie groups/algebras, e.g., in \cite{DK}, and note that virtually all results have two versions:
one for $\gg$ and one for $G$! In some sense, which will become clear in \cite{CFMb,CFMc}, the theory of PMCTs and DMCTs developed 
here makes this precise.

Recall that associated to an invariant scalar product on $\gg$ there is the so-called Cartan 3-form
\[ \phi(X,Y,Z):=\langle X,[Y,Z]\rangle, \]
where $X$, $Y$ and $Z$ are any left-invariant vector fields on $G$. This is a closed 3-form
on $G$ which serves as a background. The Cartan-Dirac structure is the $\phi$-twisted Dirac structure on $G$ defined by 
\[ L_G=\{ (\xi^L+\xi^R,\frac{1}{2}(\xi^L-\xi^R): \xi\in \gg\}, \]
where $\xi^R$ and $\xi^L$ denote the right/left translations of $\xi\in\gg$, and we use the
scalar product to identify $TG$ and $T^*G$. 

The twisted Dirac structure $L_G$ integrates to a twisted presymplectic  groupoid $\G$ which, as a Lie groupoid,
is the action groupoid $G\ltimes G$ associated to the action of $G$ on itself by conjugation. If we assume that $G$ is connected,
this is an s-connected integration of $L_G$.
The Weinstein groupoid $\Sigma(G,L_G)$ is the action groupoid $\widetilde{G}\ltimes G$, where $\widetilde{G}$ is the 1-connected 
covering Lie group of $G$.

Now observe that the action by conjugation of $G$ on itself is proper iff $G$ is compact. Hence, we conclude that:
\begin{itemize}
\item  $L_G$ is of proper type iff it is s-proper iff it is of compact type iff $G$ is compact;
\item  $L_G$ is of strong proper iff it is strong s-proper iff it is of strong compact type iff $\widetilde{G}$ is compact iff $\gg$ is a
compact semisimple Lie algebra.
\end{itemize}

\section{Constructions of PMCTs}
\label{sec:constructions}

We now present various constructions of PMCTs. In particular, our discussion of quotients  will show how a well-known question
in Symplectic Topology leads to the construction of  Poisson manifolds of strong compact type (which are not symplectic manifolds).

\subsection{Products}
Recall that a smooth function $c$ on a Poisson manifold $(M,\pi)$ is called a Casimir
if $X_c=\pi^\sharp(\d c)\equiv 0$, i.e., if it Poisson commutes with any other function.
If $(\G,\Omega)\tto M$ is any s-connected symplectic groupoid integrating $(M,\pi)$,
then $c$ is a Casimir iff it is constant on each orbit, i.e., if $c\circ s=c\circ t$.

Now let $(M_1,\pi_1)$ and $(M_2,\pi_2)$ be two Poisson manifolds and let $c\in C^\infty(M_1)$ be a Casimir. 
The \textbf{$c$-warped product} of $M_1$ and $M_2$ is the manifold $M_1\times M_2$ equipped with the Poisson
bivector $\pi=\pi_1\oplus c\pi_2$. We will denote the $c$-warped product of $M_1$ and $M_2$ by $M_1\ltimes_c M_2$.
When $c\equiv 1$ we write instead $M_1\times M_2$ and call it simply the product.

It is easy to check that for any integrable Poisson manifolds $(M_1,\pi_1)$ and $(M_2,\pi_2)$ and any non-vanishing Casimir
$c\in C^\infty(M_1)$:
\begin{itemize}
 \item if  $(\G_i,\Omega_i)\tto M_i$ are (Hausdorff) s-connected integrations of $(M_i,\pi_i)$, then 
 \begin{equation}\label{eq:warped}
  (\G_1\times\G_2,\Omega_1 \oplus\frac{1}{\widetilde{c}} \Omega_2)
 \end{equation} 
 is a (Hausdorff) s-connected integration of $M_1\ltimes_c M_2$;
 \item in particular, the canonical integration of $M_1\ltimes_c M_2$ is 
 \[ (\Sigma(M_1\ltimes_c M_2),\Omega_{\Sigma})=(\Sigma(M_1)\times\Sigma (M_2),\Omega_{\Sigma(M_1)} \oplus\frac{1}{\widetilde{c}} \Omega_{\Sigma(M_2)}). \]
 \end{itemize}

\begin{corollary}\label{cor:warped}
For any Poisson manifolds $(M_1,\pi_1)$
and $(M_2,\pi_2)$ and any non-vanishing Casimir $c\in C^\infty(M_1)$,
the $c$-warped product $M_1\ltimes_c M_2$ is of (strong) $\cC$-type iff each factor $(M_i,\pi_i)$ is of
(strong) $\cC$-type, for any $\cC$ in (\ref{types}).
\end{corollary}

We illustrate Corollary \ref{cor:warped} with the following:
\begin{example}\label{ex:product}
Consider the product Poisson manifold $(S\times \mathfrak{g}^*,\w\oplus \pi_{{\lin}})$, 
where $(S,\w)$ is a symplectic manifold. Then we see that $S\times \mathfrak{g}^*$ is:
\begin{itemize}
\item of proper type iff $\gg$ is a Lie algebra of compact type;
\item of s-proper type iff $\gg$ is a Lie algebra of compact type and $S$ is compact;
\item of strong proper type iff $\gg$ is a compact semisimple Lie algebra and $S$ has finite fundamental group;
\item of strong s-proper type iff $\gg$ is a compact semisimple lie algebra and $S$ is compact and has finite fundamental group;
\item never of compact type.
\end{itemize}
\end{example}

Corollary \ref{cor:warped} also holds in the Dirac setting. The warped product of Dirac structures is defined in a manner analogous to that
of Poisson structures. If we consider integrable Dirac structures, then allowing in (\ref{eq:warped}) the factors to
 be presympectic groupoids we obtain a presymplectic groupoid integrating the warped product. In Example \ref{ex:product} if we replace
 $\w$ by a presymplectic form leads to the same conclusions.

\subsection{Submanifolds} Recall that a Poisson submanifold of $(M, \pi)$ is a submanifold
$N\subset M$ such that $\pi_x\in \wedge^2T_xN$, for all $x\in N$. Equivalently, the Poisson bracket on $M$ descends to $N$ via the restriction map $C^{\infty}(M)\to C^{\infty}(N)$. In this case $N$ intersects the symplectic leaves $S$ of $M$ in opens inside the leaf, and the symplectic leaves of $N$ are the connected components of these intersections. One says that $N$ is saturated if each of its leaves is also leaf of $(M, \pi)$.

The integrability and/or properness of Poisson submanifolds is a subtle issue. Already in the case of the symplectic $\S^2$, which is of
strong compact type, removing a point gives rise to an open $U\subset \S^2$, which is a Poisson submanifold, 
but not of strong compact type. It is still of strong proper type, but even this fails if we remove two points instead. 
A more interesting phenomenon occurs in the case of the sphere $\S_{\gg^*}$ inside the dual $\gg^*$ of a simple Lie algebra
of compact type (with respect to the metric induced by the Killing form): this is a saturated Poisson submanifold and
while $\gg^*$ is of strong s-proper type, the sphere $\S_{\gg^*}$ is never integrable, except when $\gg=\mathfrak{su}(2)$. 

For a general Poisson submanifold $N$ of $(M, \pi)$, 
the restriction $T^{*}_{N}M$ of the cotangent algebroid to $N$ is a new Lie algebroid,
and the restriction $r: T^{*}_{N}M\rightarrow T^*N$ is a Lie algebroid morphism.
Denoting by  $\Sigma_{N}(M)$ the canonical groupoid associated to 
$T^{*}_{N}M$, it follows that there is a diagram of groupoid morphisms \cite{CF2}
 \[\begin{xy}
 (40,20)*+{\Sigma_N(M)}="bb";(60,20)*+{\Sigma(M)|_{N}}="a";%
 (60,10)*+{\Sigma(N)}="dd";%
 {\ar "bb";"a"};{\ar^{r} "bb"; "dd"};
 \end{xy}
 \]
Assuming that $(M, \pi)$ is integrable, then $T^{*}_{N}M$ is integrable as well, since $\Sigma(M)|_{N}$ integrates it.
Hence, the horizontal map becomes a smooth morphism, but it may fail to be an isomorphism since
by restricting $\Sigma(M)$ to $N$ one may destroy the 1-connectedness of the s-fibers.
This problem does not arise when $N$ is saturated. Therefore, assuming that $N$ is saturated and $(M, \pi)$ is of strong proper type,
we see that $\Sigma_{N}(M)\cong\Sigma(M)|_{N}$ is proper. However, $N$ may still fail to be integrable.

We see that the most favorable situation is when $(M, \pi)$ is of strong proper type, $N$ is saturated, and $\Sigma(N)$ smooth
and Hausdorff. Then, since $\Sigma(N)$ is a quotient of $\Sigma(M)|_{N}$, one obtains that $N$ is of strong proper type. 
This situation is rather exceptional but, as we shall prove in \cite{CFMc}, PMCTs admit two canonical stratifications whose 
strata are saturated \emph{regular} Poisson submanifolds which fit into our discussion, i.e, they are PMCTs as well. 

One can also consider other kinds of submanifolds in Poisson Geometry. For example, in the case of the so-called 
Poisson-Lie submanifolds $N\subset M$ (see \cite{CF2}) one has an embedding $\Sigma(N)\to \Sigma(M)$, and so it is easier to
find conditions under which $\cC$-types are inherited. Examples of Poisson-Lie submanifolds 
include the fixed point sets of proper Poisson actions and Poisson transversals \cite{FM}.

\subsection{Quotients}

Let $(M,\pi)$ be a Poisson manifold and let $G\times M\to M$ be an action of a Lie group
$G$ by Poisson diffeomorphisms. If the action is free and proper, then $M/G$ inherits a 
unique Poisson structure $\pi_\red$ for which the quotient map $M\to M/G$ is a Poisson map. We refer to $(M/G,\pi_\red)$ as the {\bf Poisson
reduced space}.

First, we recall from \cite{FOR} that if $(M,\pi)$ is integrable,
then there is a lifted action $G\times\Sigma(M)\to\Sigma(M)$ by symplectic groupoid
automorphisms, which is Hamiltonian with a moment map $J:\Sigma(M)\to \gg^*$ which 
is a groupoid cocycle:
\[ J(g_1 g_2)=J(g_1)+J(g_2), \quad (g_1,g_2)\in\G^{(2)}.\]
This cocycle is exact iff $G$ acts on $(M,\pi)$ in a Hamiltonian fashion. More precisely, $J=s^*\mu-t^*\mu$, where $\mu$ is a moment
map for the action. In such a situation, if $(\G,\Omega)$ is an s-connected
symplectic integration of $(M,\pi)$, then $J$ descends to a map $J_\G:
\G\rightarrow \gg^*$ which is both a groupoid cocycle and a Poisson morphism. So, if, additionally, $G$ is connected, its action
lifts to a Hamiltonian action on $(\G,\Omega)$ by groupoid automorphisms. 
Moreover, if the $G$-action on $M$ is free and proper, so is any of the lifted actions. 

In this case the {\bf symplectic quotient}
 \[(\G//G,\Omega_{\red}):=(J_\G^{-1}(0)/G,\Omega_{\red})\]
 is a symplectic integration of the 
Poisson reduced space $(M/G,\pi_\red)$. In general, this Lie groupoid need not be s-connected and, if it is, it may not be source 1-connected.
We know that upon passing to the s-connected groupoid properness may be lost. Therefore, we can only conclude:

\begin{proposition}\label{Poiss-quot-proper-type}
Let $(M,\pi)$ be a Poisson manifold and let $G\times M\to M$ be a free and proper action 
of a Lie group $G$ by Poisson diffeomorphisms. If $(M,\pi)$ is
\begin{enumerate}[(a)]
 \item either of strong s-proper (respectively, strong compact) type,
 \item or of s-proper (respectively, compact) type and $G$ is a connected group acting on $M$ in a Hamiltonian fashion,
\end{enumerate}
then the reduced space $(M/G,\pi_\red)$ is of s-proper (respectively, compact) type.
\end{proposition}

\begin{example}
Consider $\R\times \T^{2n-1}$ with the symplectic form which pulls back to the standard one in $\R^{2n}$. According to Example
\ref{ex:product} this
is an s-proper Poisson manifold. The $\S^1$-action by translations on the second factor is a free and proper Hamiltonian action.
Therefore  $\R\times \T^{2n-2}=(\R\times \T^{2n-1})/\S^1$ is a Poisson manifold of s-proper type.
\end{example}

\begin{example}
Consider the lifted cotangent action of a Lie group $G$ on its cotangent bundle $(T^*G,\omega_\can)$. This is a Poisson action which
is proper and free, and the Poisson reduced space $T^*G/G$ is isomorphic to $\gg^*$ with its linear Poisson structure. 
Note that $T^*G$ being symplectic, is always proper, and we can make it even strong proper if we choose $G$ to be 1-connected. 
On the other hand, as we saw in Section \ref{sec:linear:Poisson}, $\gg^*$ is proper if and only if $\gg$ is of compact type.
\end{example}

Neither properness nor strong properness/compactness pass to quotients in general. We need to work in a more restricted setting to guarantee that such properties descend (e.g., requiring $G$ to be a compact).

\subsection{Reductions of Hamiltonian (pre)symplectic spaces}
\label{sub-sub:Hamiltonian}
Let $(Q,\w)$ be a symplectic manifold and let $G$ be a connected Lie group acting in a Hamiltonian fashion with an equivariant moment map
$\mu: Q\to \gg^*$. We shall refer to $(Q,\w,G,\mu)$ as a {\bf Hamiltonian $G$-space}. Having a symplectic manifold with a moment map, one gets much more detailed information about the Poisson reduced space $(M=Q/G,\pi_\red)$, including an explicit description of a symplectic integration. 
%
%
%
%

First, recall that the symplectic leaves of $(M,\pi_\red)$  are the symplectic quotients:
\[ \mu^{-1}(\xi)/G_{\xi}\cong \mu^{-1}(\mathcal{O}_{\xi})/G \subset Q/G= M,\]
where $\mathcal{O}_{\xi}$ is the coadjoint orbit through $\xi\in\gg^*$ and $G_\xi$ the isotropy group at $\xi$ for the coadjoint action of $G$. 
Thus, the leaf space of $(M, \pi_\red)$ is homeomorphic to the leaf space $\gg^*/G$ of the coadjoint action. 
Also, the isotropy Lie algebras of $(M, \pi_\red)$ are isomorphic to the isotropy Lie algebras $\gg_{\xi}$, $\xi\in \gg^*$. 
Hence, one can expect that properness of $(M, \pi_\red)$ will hold, provided one requires that $G$ be compact. In fact:

\begin{proposition}\label{Ham-quot-proper-type} Let $(Q,\w,G,\mu)$ be a Hamiltonian $G$-space and assume that $G$ is compact and acts freely.
Assume further that $Q$ is connected and the fibers of the moment map are connected (respectively, 1-connected). Then:
\begin{enumerate}[(i)]
\item $(M, \pi_\red)$ is proper (respectively, strong-proper);
\item $(M, \pi_\red)$ is s-proper (respectively, strong s-proper) if the moment map is proper; 
\item $(M, \pi_\red)$ is of compact (respectively, strong-compact) type iff $G$ is a finite group (hence $M$ is symplectic). 
\end{enumerate}
\end{proposition}

\begin{proof}
First, we construct a symplectic groupoid which integrates $(M, \pi_\red)$. We start with the subgroupoid $Q\times_{\mu} Q\tto Q$
of the pair groupoid of $Q$ consisting of pairs $(q_1, q_2)$ with $\mu(q_1)= \mu(q_2)$.  The Lie group $G$ acts diagonally on this subgroupoid by groupoid automorphisms and the action is proper and free. Hence, the quotient yields the Lie groupoid:
\[
\mathcal{G}:= (Q\times_{\mu} Q)/G \tto Q/G= M,
\]
Now $\G$ carries a multiplicative symplectic form $\Omega$: the form
\[
 \textrm{pr}_{1}^{*}\omega - \textrm{pr}_{2}^{*}\omega \in \Omega^2(Q\times_{\mu}Q)
\]
 is basic with respect to the diagonal action of $G$ and 
hence descends to a 2-form $\Omega$ on $\mathcal{G}$. Since this form is closed and multiplicative, so is $\Omega$. 
One can check directly that $\Omega$ is non-degenerate. Actually, $(\mathcal{G}, \Omega)$ is a symplectic quotient: the diagonal action
of $G$ on $(Q\times Q, \omega\oplus-\omega)$ is Hamiltonian with moment map
\[ \widetilde{\mu}: Q\times Q\to \gg^*, \quad \widetilde{\mu}(q_1, q_2)= \mu(q_1)- \mu(q_2). \]
By construction
\[ \mathcal{G}= \widetilde{\mu}^{-1}(0)/G= (Q\times Q)//G\]
and $\Omega$ is the reduced symplectic form. 

One easily checks  that $(\mathcal{G}, \Omega)$ integrates $(M , \pi_\red)$ since the target map is a Poisson map.
When $Q$ is connected and the fibers of the moment map are connected (respectively, 1-connected),
the s-fibers of $(\mathcal{G}, \Omega)$ are connected (respectively, 1-connected), so items (i) and (ii) follow.
Finally, item (iii) follows from the fact that Hamiltonian actions of non-finite Lie groups on compact symplectic manifolds are never free. 
\end{proof}

A similar discussion applies in the presymplectic case, when one looks at a compact Lie group
$G$ acting in a free, proper, and Hamiltonian fashion  on a presymplectic manifold $(Q, \omega)$. Here the Hamiltonian condition
is the existence of an equivariant moment map $\mu\colon Q\rightarrow \gg^*$ satisfying:
\begin{equation}\label{moment-map-cond}
 \d\mu_v=i_{\rho(v)}\w.
\end{equation}
Due to the non-degeneracy of $\omega$, this condition does not determine the infinitesimal generators $\rho(v)$. Still, on the quotient
$M= Q/G$ one now has the the push-forward Dirac structure $L_{\mathrm{red}}= p_*(L_{\omega})$. In this case, we
call $(M, L_{\mathrm{red}})$ the \textbf{Dirac reduced space} of the Hamiltonian presymplectic space $(Q, \omega)$.
Its  presymplectic leaves are the presymplectic reduced spaces $\mu^{-1}(\xi)/G_{\xi}$. Proposition \ref{Ham-quot-proper-type} and its proof
still hold, but with ``Poisson" replaced by ``Dirac" and ``symplectic'' replaced by ``presymplectic''.  We refer to \cite{BF,BCWZ}
for more details on Dirac reduction and integrability.

The case of ($\phi$-)twisted Dirac structures requires a bit more care. We now deal with $(Q, \omega)$, where $\omega$ is a 2-form 
such that $\d\omega=\phi$. The Lie group $G$ acts in a free, proper, and Hamiltonian fashion (see \cite{BCWZ}).  
The moment map condition (\ref{moment-map-cond}), combined with the invariance of $\omega$, imply that $\phi= d\omega$ is basic and so it descends to a 3-form $\phi_M$ on $M$ which is closed (but, in general, no longer exact). The push-forward $L_{\mathrm{red}}= p_*(L_{\omega})$ is now a $\phi_M$-twisted Dirac structure and the analogue of Proposition \ref{Ham-quot-proper-type} holds.

The previous discussion can be generalized in various directions, whenever some notion of  ``generalized moment map'' is available.
This  encompasses the case of Poisson-valued moment maps of \cite{MiWe}, including the particular case of moment maps with values 
in the dual of a Poisson Lie group \cite{WeLu}, or the Lie group valued moment map of quasi-Hamiltonian spaces \cite{AMM}. The latter 
are specially interesting, since they will allow us to produce examples of Poisson manifolds of strong-compact type which are not symplectic. Moreover, in the next paper in this series \cite{CFMb} we will consider even more general versions of Hamiltonian spaces, which play a crucial role in the relationship between PMCTs, \emph{complete isotropic realizations} and \emph{symplectic gerbes}.

\subsection{Circle actions and $\S^1$-quasi Hamiltonian spaces}
\label{sub-sub:Quasi-Hamiltonian}
We extend the discussion of the previous section to $\S^1$-quasi Hamiltonian spaces. Historically, these were the first
examples of ``generalized moment maps''. They appeared in the work of McDuff on fixed points of symplectic
circle actions. We can use them to produce Poisson manifolds of strong-compact type, and 
this will bring us to a classical problem in Symplectic Topology: do there exist closed symplectic manifolds 
with symplectic free circle actions whose orbits are contractible? (see \cite{McDS}, pp. 152).

Recall that an $\S^1$-action on a (connected) symplectic manifold $(Q, \omega)$ is called {\bf quasi-Hamiltonian} if it admits an 
equivariant $\S^1$-valued moment map, i.e., an $\S^1$-invariant map $\mu: Q\to \S^1$
satisfying 
\[
 i_X(\omega)= \mu^*(\d\theta),
\]
where $X$ is the infinitesimal generator of the symplectic action and $\d\theta$ 
is the standard 1-form on $\S^1$. This is the  generalized moment map of \cite{McD}.
If the action is free,  then the exact same arguments as in the proof of Proposition \ref{Ham-quot-proper-type} yield
a symplectic groupoid integrating the Poisson reduced space $(Q/\S^1,\pi_\red)$:
\[ \mathcal{G}:= (Q\times_{\mu} Q)/\S^1 .\]
The s-fibers of $\G$ coincide with the fibers of $\mu$ and we deduce:

\begin{corollary}  
If $(Q, \omega,\S^1,\mu)$ is a connected, free quasi-Hamiltonian $\S^1$-space, and the moment map has connected (respectively, 1-connected) fibers, 
then the Poisson reduced space $(Q/\S^1,\pi_\red)$ is of compact (respectively, strong compact) type.
\end{corollary}

How can one ensure that the fibers of the moment map are connected or even 1-connected? In this direction one has the following proposition,
showing that finding such examples is non-trivial since it relates to the aforementioned Symplectic Topology problem. 

\begin{proposition}  Given a free symplectic $\S^1$-action on a connected, compact symplectic manifold $(Q, \omega)$, one has that:
\begin{enumerate}[(i)]
\item in general, one can perturb and rescale $\omega$ to a new invariant symplectic form $\omega'$ such that 
$(Q, \omega')$ is a quasi-Hamiltonian space whose moment map $\mu$ has connected fibers, and
\item if the fibers of $\mu$ are 1-connected, then the orbits of the $\S^1$-action must be contractible.
The converse holds, provided the symplectic leaves of $Q/\S^1$ are 1-connected.
\end{enumerate}
\end{proposition}

\begin{proof} 
In order to prove (i), we invoke \cite[Lemma 1]{McD}  which allows us to assume that $(Q, \omega)$ is quasi-Hamiltonian 
(modulo small perturbations of $\omega$). In general, the fibers of its moment map $\mu: Q\to \S^1$ will not be connected (e.g.,  $\S^1\times \S^1$ with
the standard action of $\S^1$ on the first factor,
$\omega= 2d\theta_1\wedge d\theta_2$, and $\mu(z_1, z_2)= z_{1}^{2}$). However, this  can be easily fixed as follows.
The connectedness of the fibers is equivalent to the fact that 
\[ \mu_*: \pi_1(Q)\to \pi_1(\S^1)\cong \mathbb{Z}\]
is surjective. In general, the image of $\mu_{*}$ cannot be zero since otherwise $\mu$ would admit a
lift $\widetilde{\mu}: Q\to \mathbb{R}$ that would be a moment map of $(Q, \omega)$ in the classical sense, contradicting the freeness of the action.
Hence, the image is a cyclic group $k\mathbb{Z}$ with $k\geq 1$ an  integer, and this implies that $\mu$ admits a lift $\widetilde{\mu}: Q\to \S^1$
along the covering map $\S^1\to \S^1$, $z\mapsto z^k$. Then we just replace $\omega$ by $\omega'= \frac{1}{k}\omega$ (but keep the same $\S^1$-action!). 

Now, to prove (ii), assume that we have a free quasi-Hamiltonian action with moment map $\mu$ having connected fibers.
Note that $\mu$ descends to a fibration $\overline{\mu}: M\to \S^1$ and the symplectic leaves of $M$ are precisely the fibers of $\overline{\mu}$.
If the fibers of $\mu$ are 1-connected, then it  follows from the homotopy long exact sequence that $\mu_*: \pi_1(Q)\to \pi_1(\S^1)$ is an isomorphism.
Since any orbit loop in $Q$, i.e., any loop of type $\S^1\ni z\mapsto q\cdot z\in Q$,
is mapped by $\mu$ into the constant loop, the first claim follows. 

For the converse, fix a fiber $F$ of $\mu$. The homotopy long exact sequence associated to $\mu$ implies that
the canonical map $i_*: \pi_1(F)\to \pi_1(Q)$ is injective, so a loop inside $F$ is contractible in $F$ iff it is contractible
in $Q$. Now, the homotopy long exact sequence applied to $F\to F/\S^1$,
together with the extra-assumption that $F/\S^1$ is 1-connected, implies that every class in $\pi_1(F)$ is represented
by a multiple of an $\S^1$-orbit. But, by assumption, orbits are contractible in $Q$ and hence contractible in $F$.
This proves the 1-connectedness of the fiber $F$.
\end{proof}

The existence of free symplectic actions with contractible orbits (on compact manifolds) 
was finally proved by Kotschick in \cite{Ko}, building on results of \cite{FGM} to produce quasi-Hamiltonian $\S^1$-spaces. 
The construction relies on deep properties of the moduli space of marked hyper-K\"ahler K3 surfaces. 
Full details, emphasizing the Poisson nature of the construction, can be found in \cite{Mar},
which also takes care of some issues that were overlooked in \cite{Ko}. 

We deduce, in particular:

\begin{corollary}\label{K3} There exists a 5-dimensional regular Poisson manifold of strong compact type,
with symplectic leaves K3-surfaces fitting into a fibration over $\S^1$. 
\end{corollary}

\subsection{Fusion product}

The notion of quasi-Hamiltonian $\S^1$-space has a generalization for arbitrary Lie groups \cite{AMM}. Here we are only interested 
in the case where $G$ is a torus, where the generalization is straightforward:  a quasi-Hamiltonian $\T^n$-space is a $\T^n$-symplectic
manifold $(Q,\w)$ with a $\T^n$-invariant moment map $\mu:Q\rightarrow \T^n$ satisfying the map condition:
\[  i_{X}(\omega)= \mu^*(\d\theta_X), (i=1,\dots,n)\]
where $X\in\R^n$ is any element of the Lie algebra of $\T^n$ and $\d\theta_X$ is the right-invariant 1-form on $\T^n$ which at
the identity takes the value $Y\mapsto \langle Y,X\rangle$.

Exactly as in the case of circle actions, we have:
\begin{corollary}  
If $(Q, \omega,\T^n,\mu)$ is a connected, free quasi-Hamiltonian $\T^n$-space, and the moment map has connected (respectively, 1-connected) fibers,
then the Poisson reduced space $(Q/\T^n,\pi_\red)$ is of compact (respectively, strong compact) type.
\end{corollary}

Likewise,   the 1-connectivity of the fibers of the $\T^n$-valued moment map implies that the orbits of any circle subgroup of $\T^n$ are contractible.
The converse also holds provided the fibers of the moment map are connected and the leaves of the Poisson reduced space are 1-connected.

There is an operation on quasi-Hamiltonian spaces that allows to construct new Poisson manifolds of strong compact type,  namely
the {\bf fusion product} of two quasi-Hamiltonian spaces (see \cite{AMM}). Given quasi-Hamiltonian spaces  $(Q_1,\w_1,\T^n,\mu_1)$ 
and $(Q_2,\w_2,\T^n,\mu_2)$ its fusion is the quasi-Hamiltonian space 
\[ (Q_1\times Q_2,\mathrm{pr}_1^*\w_1-\mathrm{pr}_2^*\w_2,\T^n,\mu),\] 
where $\T^n$ acts diagonally with fusion moment map:
\[\mu_1\mu_2^{-1}: Q_1\times Q_2\rightarrow \T^n,\quad (x_1,x_2)\mapsto \mu_1(x_1)\mu_2^{-1}(x_2).\]
The fusion product is a fundamental operation that can be defined for the more general \emph{symplectic gerbes} in \cite{CFMc}.

In our setting, the moment-fibers of the fusion product and the leaves of its Poisson 
reduced space can be easily described:

\begin{proposition}\label{fusion-fibers} 
Let $(Q_1,\w_1,\T^n,\mu_1)$ and $(Q_2,\w_2,\T^n,\mu_2)$ be compact, connected
quasi-Hamiltonian spaces. If the action of $\T^n$ on $Q_1$ is free, then:
 \begin{enumerate}[(i)]
  \item the moment-fiber of the fusion product  is a fibration over $Q_2$ with fiber diffeomorphic to the $\mu_1$-fiber;
 \item each symplectic leaf of the fusion Poisson reduced space is a fibration over $Q_2$ with fiber 
 diffeomorphic to the fiber of $\overline{\mu_1}:M_1=Q_1/\T^n\to \T^n$.
\end{enumerate}
\end{proposition}
\begin{proof}
The diagonal action on the fusion product is also free and therefore the moment map is a fibration.  The fiber over the unit 
\begin{equation}\label{e-fiber}\mu_1\mu_2^{-1}(e)=\{(x_1,x_2)\in Q_1\times Q_2\ |  \ \mu_1(x_1)=\mu_2(x_2)\}
 \end{equation}
is a submanifold of $Q_1\times Q_2$. We claim that the restriction of the second projection
\[\mathrm{pr}_2: \mu_1\mu_2^{-1}(e)\rightarrow Q_2\]
is a surjective submersion: the surjectivity of $\mathrm{pr}_2$ is an immediate consequence of the surjectivity of $\mu_1$. On the other hand, we have:
\[ T_{(q_1,q_2)}\mu_1\mu_2^{-1}(e)=\{(v,w)\in T_{q_1}Q_1\times T_{q_2}Q_2:\d_{q_1}\mu_1\cdot v-\d_{q_2}\mu_2\cdot w=0\}. \]
Given any $w\in T_{q_2}Q_2$, since $\mu_1$ is a submersion, we can find $v\in T_{q_1}Q_1$ such that $\d_{q_1}\mu_1\cdot v=\d_{q_2}\mu_2\cdot w$. This proves the surjectivity of the differential of $\mathrm{pr}_2$. By (\ref{e-fiber}), the fiber of $\mathrm{pr}_2$
is the $\mu_1$-fiber, so (i) holds.

In order to prove (ii), observe that the fusion moment map descends to the Poisson reduced space $Q_1\times Q_2/\T^n$ giving a fibration:
\[ \overline{\mu_1\mu_2^{-1}}: Q_1\times Q_2/\T^n\to \T^n. \]
The fibers of this map are the symplectic leaves of the Poisson reduced space. The fiber over the unit can be identified with:
\[ \left(\overline{\mu_1\mu_2^{-1}}\right)^{-1}(e)\cong \{([x_1],x_2)\in M_1\times Q_2\ |  \ \mu_1(x_1)=\mu_2(x_2)\}.\]
Since $\mu_1$ also descends to a fibration $\overline{\mu}_1: M_1\rightarrow \T^n$, the previous argument implies that the restriction of 
$\mathrm{pr}_2$ to the fiber over the unit is a fibration with fiber diffeomorphic to the fiber of $\overline{\mu}_1$.
 \end{proof}

Proposition  \ref{fusion-fibers} shows that out of a Poisson manifold of strong compact type coming form a free quasi-Hamiltonian $\T^n$-space with 
1-connected moment-fibers, we can produce plenty of Poisson manifolds of compact type:

\begin{corollary}\label{cor:fusion} 
Let $(M_1,\w_1,\T^n,\mu_1)$ and $(M_2,\w_2,\T^n\mu_2)$ be compact (connected) quasi-Hamiltonian spaces. 
Assume further that the action of $\T^n$ on $M_1$ is free, $\mu_1$ has 1-connected fibers, and $M_2$ is 1-connected. Then the Poisson
reduced space of the fusion product is a Poisson manifold of strong compact type.
 \end{corollary}
 
Note that while quasi-Hamiltonian spaces of the first type are hard to produce, there are plenty of quasi-Hamiltonian spaces of the second type:
any Hamiltonian $\T^n$-space becomes quasi-Hamiltonian upon composing the (classical) moment map with the exponential map of $\T^n$ (see \cite{AMM}).

\subsection{Gauge equivalences} 
\label{sub-sub:Gauge Equivalences} 
Two Poisson structures $\pi_1$ and $\pi_2$ on $M$ are said to be {\bf gauge equivalent} if they 
have the same foliation and there exists a closed 2-form $\a\in\Omega^2(M)$ such that for any leaf $S$,
the symplectic forms on $S$ induced by $\pi_1$ and $\pi_2$ differ by the pullback of $\a$ to $S$ (\cite{SW}).
This condition becomes complicated when expressed in terms of the bivectors $\pi_i$. On the other hand, it is easy to express it in terms of 
the associated Dirac structures $L_{\pi_i}$. In fact, gauge equivalence finds its natural setup in Dirac geometry (\cite{BR}).

The key-remark  is that the space of Dirac structures on a manifold $M$ admits as symmetries the semidirect
product $\mathrm{Diff}(M)\ltimes \Omega^2_\mathrm{cl}(M)$, where the group $\mathrm{Diff}(M)$ acts on closed forms by Lie derivative.
The diffeomorphisms act on Dirac structures by push-forward, while a closed 2-form $\alpha\in \Omega^2_\mathrm{cl}(M)$ acts 
on a Dirac structure $(M,L)$ by the {\bf gauge transformation}:
\[ \a\cdot L:=\{(X,\b+i_X\a): (X,\b)\in L\}. \]
Clearly, $\a\cdot L$ is a Dirac structure with the same presympectic foliation as $L$, while the presymplectic form on each leaf of $\a\cdot L$ 
is the result of adding to the initial presymplectic form the pullback of $\a$ to the leaf.

Gauge transformations preserve integrability of Dirac structures. In fact, if $(\G,\Omega)$ is any presymplectic groupoid integrating 
the Dirac structure $(M,L)$ and $\alpha\in \Omega^2_\mathrm{cl}(M)$, then the same Lie groupoid 
with the modified presymplectic form
\[ \a\cdot\Omega:=\Omega+s^*\a-t^*\a \]
gives a presymplectic groupoid integrating $(M,\a\cdot L)$ (see \cite{BR}). Therefore:

\begin{proposition}
For any of the  compactness types (\ref{types}) the notion of $\mathcal{C}$-type and strong $\mathcal{C}$-type is invariant under gauge
equivalences of Dirac structures. 
\end{proposition}

\begin{example}
\label{ex:gauge:equivalent}
Consider two Hamiltonian (pre)symplectic spaces $(Q,\omega_1,G,\mu)$ and $(Q,\omega_2,G,\mu)$, with the same $G$-action and moment
map $\mu:Q\to\gg^*$, so only the (pre)symplectic forms are different. The difference $\omega_2-\omega_1$ is closed and is annihilated
by the infinitesimal generators of the action, so it is $G$-basic: $\omega_2-\omega_1=p^*\a$, 
where $p: Q\to M=Q/G$ and $\alpha\in \Omega^2_\mathrm{cl}(M)$. Denoting by $L_1$ and $L_2$ the reduced Dirac structures on $M$,
it then follows that they are gauge equivalent:
\[ L_2= \alpha \cdot L_1.\]
\end{example}

\subsection{The local linear model}
\label{sec:local:model}

We illustrate now many of the previous constructions by combining them to build a Poisson/Dirac structure which will 
turn out to be a local model for PMCTs and DMCTs (see \cite{CFMc}).

%
%
%
The starting data for the local linear model consists of:
\begin{itemize}
\item a (pre)symplectic manifold $(S, \omega)$;
\item a principal $G$-bundle $p: P\to S$, where $P$ is connected and $G$ is a (possibly disconnected) compact Lie group;
\item a principal bundle connection $\theta\in \Omega^1(P, \gg)$.
\end{itemize}
We endow $P\times \gg^*$ with with the presymplectic form
\[
\omega^{\theta}_{\lin}= p^*\omega-\d\langle \theta,\cdot\rangle \in \Omega^2(P\times \gg^*).
\]
Using the coadjoint action, we have a diagonal action of $G$ on $(P\times \gg^*,\omega^{\theta}_{\lin})$. This action is free
and Hamiltonian with moment map $\mu=\mathrm{pr}_2:P\times \gg^*\to \gg^*$. The Dirac reduced space is 
the associated  bundle $P\times_G \gg^*$ and it carries  a reduced Dirac structure $L_{\lin}^\theta$. We call it the {\bf local linear model}.
Note that $S=P\times_G\{0\}\subset P\times_G \gg^*$ is a (pre)symplectic leaf of this reduced Dirac structure, with (pre)symplectic form $\omega$.

Our discussion in Section \ref{sub-sub:Hamiltonian}, gives:

\begin{corollary}
\label{cor:local:model}
If $P$ is a connected principal $G$-bundle over the presymplectic manifold $(S, \omega)$ and $G$ is a compact Lie group,
then the local linear model Dirac structure $L_{\lin}^\theta$ on $(P\times_G \gg^*,L_{\lin}^\theta)$ is:
\begin{enumerate}[(i)]
\item always of proper type;
\item of strong proper type if $P$ is 1-connected;
\item of s-proper type if $P$ is compact;
\item of strong s-proper type if $P$ is compact and 1-connected;
\item never of compact type.
\end{enumerate}
\end{corollary}

In our study of local normal forms of PMCTs and DMCTs \cite{CFMc} we will see that the necessary conditions stated above for the various
compactness types are actually sufficient.

\begin{remark}
Choosing a different connection 1-form $\theta'$ results in a different Dirac structure on the reduced space.
However, Example \ref{ex:gauge:equivalent} shows that $L_{\lin}^\theta$ and  $L_{\lin}^{\theta'}$ are gauge equivalent, 
so the connection can be thought of as auxiliary data.
\end{remark} 

When $(S,\w)$ is a symplectic manifold the presymplectic form $\omega^{\theta}_{\lin}$ is actually symplectic in
a $G$-invariant neighborhood of $P\times\{0\}$ in $P\times\gg^*$, and the Poisson support of $L_{\lin}^\theta$ includes 
a neighborhood of $S\subset P\times_G \gg^*$. We call the resulting Poisson bivector the {\bf local linear model} and denote it by $\pi^\theta_\lin$.

\begin{remark}
When $(S,\w)$ is symplectic, $P$ is compact and we choose two different connections, a standard Moser argument applied in an 
appropriate $G$-invariant neighborhood $P\times\{0\}\subset U\subset P\times\gg^*$ shows that we can find a $G$-invariant 
symplectomorphism between $\omega^{\theta}_{\lin}|_U$ and $\omega^{\theta'}_{\lin}|_U$, which is the identity on $P$.
It follows that the resulting local linear models $\pi^\theta_\lin$ and $\pi^{\theta'}_\lin$ are not only gauge equivalent 
but also Poisson diffeomorphic in a neighborhood of $S$, via a diffeomorphism which is the identity in $S$.
\end{remark} 

\section{Poisson homotopy groups and strong compactness}\label{Poisson-homotopy-groups}

The ``strong'' versions of the various compactness conditions are somewhat more intrinsic
since they depend only on the canonical integration $\Sigma(M, \pi)$. For instance, 
while properness implies that the isotropy Lie algebra $\gg_x(M, \pi)$ comes from {\it some} 
compact Lie group covered by the Poisson homotopy group $\Sigma_{x}(M, \pi)$,
strong properness implies that  $\Sigma_{x}(M, \pi)$ is itself compact. 
In fact, this is the \emph{only} difference between $\mathcal{C}$-type and strong $\mathcal{C}$-type,
since we have:

\begin{proposition} A Poisson manifold $(M, \pi)$ is of strong s-proper type (respectively, strong compact type) 
iff it is of s-proper type (respectively, compact type) and all the groups $\Sigma_{x}(M, \pi)$ are compact. 
\end{proposition}

\begin{proof} If $\G$ is an s-proper integration with connected $s$-fibers,
Ehresmann's theorem implies that $s: \G\to M$ is locally trivial. It follows that the  groupoid morphism
$\Sigma(M,\pi)\to\G$ is a covering map with fibers isomorphic to a discrete subgroup of  $\Sigma_{x}(M, \pi)$, and hence finite.
We conclude that $\Sigma(M,\pi)$ is s-proper. 
\end{proof} 

The previous proposition shows that it is important to obtain characterizations of the compactness of $\Sigma_x(M, \pi)$ 
in terms of the Poisson geometry. This is actually possible as shown by the following theorem where we do not assume 
integrability of the Poisson manifold:

\begin{theorem}\label{prop-compactness-Poisson-homotopy} If $(M, \pi)$ is a Poisson manifold and $x\in M$, 
then $\Sigma_x(M, \pi)$ is a compact Lie group iff all the following properties hold:
\begin{enumerate}[(i)]
\item the fundamental group of the leaf $S$ through $x$ is finite;
\item the isotropy Lie algebra $\gg_x(M, \pi)$ is a Lie algebra of compact type;
\item the monodromy group $\cN_x$ is a lattice in $Z(\gg_x)$.
\end{enumerate}
\end{theorem}

\begin{remark}
At regular points, the isotropy Lie algebra $\gg_x$ is abelian and coincides with the conormal space $\nu_x^*(S)$.
We will see in \cite{CFMb,CFMc} that for a strong s-proper Poisson manifold the lattices $\cN_x\subset\nu_x(S)$ form a Lagrangian lattice, 
so the regular part of the symplectic foliation admits a transverse integral affine structure (see Section \ref{sec:integral:affine:foliation}); 
this will lead to yet another characterization of regular Poisson manifolds strong s-proper type in terms of properties of the symplectic foliation and 
of the monodromy groups.
\end{remark}

For the proof of the theorem we need to recall additional aspects of the monodromy groups $\cN_x$, which so far we have only discussed
in the integrable case.  

The group $\cN_x$ arises naturally when trying to compute $\Sigma_x$ directly out of $(M, \pi)$. Let us assume for a moment 
that $(M,\pi)$ is integrable, and let us look at the first terms of the homotopy long exact sequence of 
the fibration $t: s^{-1}(x)\to S$. Since the source fibers are 1-connected, we obtain:
\begin{equation}\label{hom-ex-seq-part}
\xymatrix{\ar[r]& \pi_2(S) \ar[r]^\partial & \pi_1(\Sigma_{x})\ar[r]& 0\ar[r] & \pi_1(S)\ar[r]& \pi_0(\Sigma_x) \ar[r]& 0}. 
\end{equation}
First of all, we conclude that $\pi_0(\Sigma_x)$ is isomorphic to $\pi_1(S_x)$. Next, we observe that $\Sigma_{x}^{0}$ (the connected component of the identity) integrates $\gg_x$. Hence, it is a quotient of the 1-connected Lie group $G(\gg_x)$ integrating $\gg_x$ by
some discrete subgroup $\widetilde{\cN}_x$ of the center:
\[ \Sigma_{x}^0= G(\gg_x)/\widetilde{\cN}_x .\]
The group $\widetilde{\cN}_x$ is called the {\bf extended monodromy group} of $(M, \pi)$ at $x$.  Identifying $G(\gg_x)$ with 
the universal covering space of $\Sigma_{x}^{0}$, the discrete group $\widetilde{\cN}_x$ is just $\pi_1(\Sigma_{x}^{0})$, 
realized as the kernel of the covering projection. One can also obtain $\widetilde{\cN}_x$
as the image of the so-called {\bf monodromy map}:
\begin{equation}\label{mon-map} 
\partial_x: \pi_2(S_x)\to G(\gg_x), 
\end{equation}
which is the composition of  the connecting homomorphism $\partial$ from (\ref{hom-ex-seq-part})
with the canonical inclusion $\pi_1(\Sigma_x)\subset G(\gg_x)$. 

The {\bf monodromy group} $\mathcal{N}_x$ arises as a simplification of $\widetilde{\cN}_x$ obtained by passing from 
$G(\gg_x)$ to $\gg_x$ via the exponential map. More precisely, the exponential map $\widetilde{\textrm{exp}}: \gg_x\to G(\gg_x)$ gives an isomorphism
\[ 
\widetilde{\textrm{exp}}|_{Z(\gg_x)}: Z(\gg_x){\to} Z(G(\gg_x))^0,
\]
and one defines:
\begin{equation}\label{mon-gp-gen} 
\cN_x:=\{u\in Z(\gg_x): \widetilde{\textrm{exp}}(u)\in \widetilde{\cN_x}\}.
\end{equation}

If the integrability assumption is dropped, then the monodromy map
(\ref{mon-map}) can still be defined directly in terms of the cotangent Lie algebroid of $(M, \pi)$ \cite{CF1,CF2}.
%
The extended monodromy group $\widetilde{\cN}_x$ is still defined as the image of $\partial_x$, while 
the monodromy group is given by (\ref{mon-gp-gen}). Moreover, the relevance of these two groups to integrability should be 
clear: since $\Sigma_{x}^{0}\cong G(\gg_x)/\widetilde{\cN}_x$, for this quotient to be a manifold 
$\widetilde{\cN}_x$ must be a discrete subgroup of $G(\gg_x)$, and this holds iff $\cN_x$ is a discrete subgroup of $Z(\gg_x)$.

\begin{proof}[Proof of Theorem \ref{prop-compactness-Poisson-homotopy}]
Let us set $G= G(\gg_x)$. If $\Sigma_x$ is a compact Lie group, then $\gg_x$ is compact. Moreover,
$s^{-1}(x)$ is a smooth principal $\Sigma_x$-bundle, so by the long exact homotopy sequence (\ref{hom-ex-seq-part}), we see that
$\pi_1(S)$ is finite. The discreteness of $\cN_x$ is clear since $\Sigma_x=G/\widetilde{\cN}_x$. Also, since this group is compact,
so is the closed subgroup $Z(G)/\widetilde{\cN}_x$, and then also 
\[ Z(\gg_x)/\cN_x\cong Z(G)^0/Z(G)^0\cap \widetilde{\cN}_x \subset Z(G)/\widetilde{\cN}_x .\]
Therefore $\cN_x$ must  be a lattice in $Z(\gg_x)$.

Conversely, the assumptions imply that $t: s^{-1}(x)\to S$ is a smooth principal $\Sigma_x$-bundle and that 
$\pi_0(\Sigma_x)\cong \pi_1(S)$ is finite. Hence we are left with proving that $\Sigma_{x}^{0}= G/\widetilde{\cN}_x$ is a compact Lie group.
By the previous discussion on integrability, $\widetilde{\cN}_x$ is discrete in $G$ and we have a short exact sequence of Lie groups
\[ 
\xymatrix{1\ar[r] & Z(G)/\widetilde{\cN}_x \ar[r] & G/\widetilde{\cN}_x \ar[r] & G/Z(G)\ar[r] & 1},
\]
where the last group is compact because the Lie algebra $\gg_x$ of $G$ is of 
compact type. Hence, to prove that the middle group is compact, we are left with proving that the first one is.
For this we use a similar short exact sequence of Lie groups
\[ 
\xymatrix{1\ar[r] & Z(G)^0/\widetilde{\cN}_x \ar[r] & Z(G)/\widetilde{\cN}_x \ar[r] & Z(G)/(Z(G)^0\cdot \widetilde{\cN}_x)\ar[r] & 1}.
\]
Note that $Z(G)/Z(G)^0$ is a finite group because $\gg_x$ is of compact type. Thus the last group in the sequence is finite. The first
group in the sequence is also compact being isomorphic to $Z(\gg_x)/\cN_x$, with $\cN_x$ a lattice. This proves the theorem.
\end{proof}

\section{Basic properties of PMCTs}
\label{sec:basic-properties}
We now turn to some basic properties of PMCTs. Many of these properties are particular cases of general properties of Lie groupoids
that are proper, s-proper or compact.
The Poisson geometric context allows one to translate them into more geometric properties.
Some of these will be mentioned below. However, many of them will only be fully addressed in 
\cite{CFMb,CFMc}, for they require some new techniques and ideas.

For now, let us start by listing the most basic properties of proper Lie groupoids:


\begin{lemma} \label{lemma-gen-pr-gpds} If $\cG\tto M$ is a proper Lie groupoid, then:
\begin{enumerate}[(i)]
\item the orbits of $\cG$ are closed embedded submanifolds;  if $\G$ is source 1-connected they have finite fundamental group; 
if $\G$ is $s$-proper they are compact;
\item the orbit space (with the quotient topology) is Hausdorff;
\item all the isotropy groups $\cG_x$ are compact;
\item the linear holonomy groups  (see (\ref{lin-hol-gpd-case})) are finite. 
\end{enumerate}
\end{lemma}

In the case of Poisson/Dirac manifolds of proper type, we see already that there are restrictions on the topology of the symplectic foliation 
and on the isotropy Lie algebras. We shall see in \cite{CFMb,CFMc} that in the Poisson/Dirac case a lot more can be said about the 
space of symplectic leaves, namely the existence of an integral orbifold structure, which does not hold for the orbit spaces of general proper groupoids.

An immediate consequence of the previous lemma is the following: 

\begin{corollary}  If $(M, \pi)$ is a Poisson manifold of proper type, then all the isotropy
Lie algebras $\gg_x(M, \pi)$ are Lie algebras of compact type.
\end{corollary}

\subsection{Basic Poisson-topological properties.}

For the next propositions it is useful to recall that for any Lie groupoid $\G$ with Lie algebroid $A$,
one has the Van Est map relating differentiable groupoid cohomology and Lie algebroid cohomology:
\[ \mathrm{VE}:H^k(\G)\to H^k(A). \]
The results of \cite{Cr1} show that:
\begin{enumerate}[(a)]
\item if the s-fibers are homologically $k$-connected, then the van Est map is an isomorphism up to degree $k$ and injective in degree $k+1$;
\item the differentiable cohomology of any proper Lie groupoid $\G$ vanishes above degree 0.
\end{enumerate} 
For the cotangent Lie algebroid of a Poisson manifold $(M,\pi)$ the Lie algebroid cohomology is the same 
as Poisson cohomology: $H^\bullet(T^*M)=H^\bullet_\pi(M)$ . Hence, we obtain:

\begin{proposition} 
\label{prop:first:cohomology}
For any Poisson manifold of strong proper type $H^{1}_{\pi}(M)=0$. 
\end{proposition}

\begin{proof} 
The symplectic groupoid $\Sigma(M,\pi)$, together with the source map, is a local trivial fibration
over $M$ whose fibers are 1-connected. Hence we have $H^1_\pi(M)\cong H^1(\Sigma(M,\pi))=0$.
\end{proof}

More generally, the previous argument applies to show that $H^1(A)=0$ for any Lie algebroid with a strong proper integration, 
in particular, for Dirac manifolds of strong proper type.

The second Poisson cohomology, however, does not vanish since we have the following analogue of non-vanishing
of the cohomology class of the symplectic form on a compact symplectic manifold:

\begin{proposition}[\cite{CF3}]
\label{prop:h2:non-vanishing}
Let $(M,\pi)$ be a Poisson manifold of strong compact type. Then $[\pi]\in H^2_\pi(M)$ is non-trivial.
\end{proposition}


\begin{proof}
Note that $\pi\ne 0$. The exactness of $[\pi]\in H^2_\pi(M)$ is equivalent to the existence of $X\in \mathfrak{X}(M)$
such that 
\[\mathfrak{L}_X\pi=\pi.\]
Since $M$ is compact the vector field $X$ integrates to a flow $\phi_t$ which exponentiates the Poisson bivector.
In particular, for each $t\in \R$, we obtain Poisson diffeomorphisms
\[\phi_t\colon (M,\pi)\rightarrow (M,\mathrm{e}^t\pi),\]
which induce Lie algebroid isomorphisms of the underlying cotangent Lie algebroids. By Lie's second 
theorem for symplectic groupoids these algebroid isomorphisms have unique lifts to  symplectic groupoid isomorphisms:
\[\Phi_t\colon (\Sigma(M,\pi),\Omega)\rightarrow (\Sigma(M,\pi),\mathrm{e}^t\Omega).
\]
The diffeomorphisms $\Phi_t$ fit into a flow. If $Y$ is the corresponding vector field, then
we must have
\[\Omega=\d i_Y\Omega,\]
which is not possible since $\Sigma(M,\pi)$ is compact.
\end{proof}

From the non-vanishing of the second Poisson cohomology we deduce another interesting property:

\begin{proposition}
\label{prop:no-fixed-points}
A Poisson manifold of strong compact type has no fixed points. 
\end{proposition}

\begin{proof} Assume that $(M, \pi)$ is of strong compact type with at least one fixed point $x_0$.
Then the symplectic groupoid $\Sigma(M,\pi)$, together with the source map, is a local trivial fibration
over $M$ whose fiber above $x_0$ is by assumption $\Sigma_{x_0}$. Since this fiber is 2-connected,
it follows that the s-fibers of $\Sigma$ are 2-connected. Hence we must have $H^2_\pi(M)\cong H^2(\Sigma(M,\pi))=0$, 
contradicting Proposition \ref{prop:h2:non-vanishing}. 
\end{proof}

The vanishing of $H^{1}_{\pi}(M)$ for a strong proper Poisson manifold has some important consequences. First, it implies that any 
action by Poisson diffeomorphism is Hamiltonian, which obvious fails for arbitrary Poisson manifolds: 

 \begin{proposition} 
Let $(M,\pi)$ be a strong proper Poisson manifold. Then every Poisson action is Hamiltonian.
\end{proposition}

\begin{proof}
It was already mentioned that the results of \cite{FOR} show that any Poisson action has a canonical 
lift to a Hamiltonian action on the Weinstein groupoid $(\Sigma(M),\Omega_\Sigma)$. The moment map for the lifted
action is a 1-cocycle in the differentiable groupoid cohomology of $\Sigma(M)\tto M$, and the Poisson action is Hamiltonian 
if and only if this 1-cocycle is exact. This happens in particular when $\Sigma(M)$ is proper.
\end{proof}

\subsection{Invariant measures}
For a Poisson manifold $(M,\pi)$ the vanishing of $H^{1}_{\pi}(M)$ guarantees the existence of densities (hence also of measures)
invariant under all Hamiltonian flows, i.e. sections $\rho$ of the density bundle 
such that:
\[ \Lie_{X_h}\rho=0,\quad \forall h\in C^\infty(M).\]
In fact, the obstruction to the existence of such a density lies precisely in $H^1_\pi(M)$, and it is the so-called modular class of 
the Poisson manifold $(M, \pi)$ \cite{We2}. 
When the modular class vanishes we say that $(M,\pi)$ is \textbf{unimodular}.  

Actually, unimodularity holds under the weakest compactness assumption, although the first Poisson cohomology group itself need not vanish:
 
\begin{proposition} 
The modular class of any proper Poisson manifold vanishes. 
\end{proposition}

\begin{proof} 
The modular class is in the image of the Van Est map $\textrm{VE}:H^1(\G)\to H^1_\pi(M)$, for any integration $\G$ (see \cite{Cr1}).
Since for a proper groupoid $H^k(\G)$ vanishes above degree 0, any proper Poisson manifold is unimodular.
\end{proof}

The situation is even better for s-proper Poisson manifolds: using an s-proper integration $(\cG, \Omega)$, one can produce right away
such an invariant density. At the level of measures, we just push-forward via the source map the Liouville measure induced by $\Omega$;
at the level of densities, that means that we define our density as the integration along the s-fibers of the Liouville density $|\Omega^n|/n!$
(of course, the reason to work with densities instead of volume forms comes from the fact that densities can be integrated over fibers, 
to give densities on the base, without any orientability assumptions). Therefore, we define the density on $M$ given by 
\[ \rho_{\mathrm{DH}}^{\cG}:= \int_{s} |\Omega^n|/n!\quad (n=\dim M=1/2\dim \G).\]
This will be called {\bf the Duistermaat-Heckman density} on $M$ induced by $\cG$.

\begin{proposition}\label{pro:fibvol} 
For any s-proper Poisson manifold $(M,\pi)$, the Duistermaat-Heckman density $\rho_{\mathrm{DH}}^{\cG}$ induced by an s-proper integration
$(\cG, \Omega)$ is an Hamiltonian invariant density for $(M, \pi)$. 
\end{proposition}

\begin{proof}
Since the source is a Poisson map, for any $h\in C^\infty(M)$ 
the Hamiltonian vector fields $X_h$ and $X_{s^*h}$
are $s$-related, and therefore 
\[\Lie_{X_h}\rho_{\mathrm{DH}}^{\cG}=\int_s \Lie_{X_{s^*h}}\frac{|\Omega^n|}{n!}=0.\]
 \end{proof}

\begin{remark}
We will show in \cite{CFMb,CFMc} that the transverse integral affine structure present on any PMCT defines another Hamiltonian invariant density $\rho_{\mathrm{Aff}}^M$. In the the s-proper case there is a polynomial Duistermaat-Heckman formula relating
the integral affine density $\rho_{\mathrm{Aff}}^M$ and the Duistermaat-Heckman density $\rho_{\mathrm{DH}}^\cG$.
\end{remark}

\subsection{Poisson cohomology ring and Poincar\'e duality pairing}

For a general Poisson manifold the Poisson cohomology ring is infinite dimensional and extremely hard to compute.
However, for an s-proper Poisson manifold the cohomology is essentially a finite dimensional object, as we now explain.

Let  $(\cG,\Omega)$ be an s-proper symplectic integration of $(M,\pi)$. Then $s:\cG\to M$ is a fibration with compact fibers,
hence it is a locally trivial fibration with structural group the diffeomorphism of the fiber. We consider the associated vector 
bundle
\[\cH^\bullet\to M\]  
whose fiber is the total de Rham cohomology of the $s$-fiber. This bundle carries a right $\cG$-action induced from the right action of $\cG$ on itself, 
i.e., $\cH^\bullet\to M$ is a representation of $\cG$. The fiberwise cup product induces a ring structure on $\Gamma(\cH^\bullet)$ and 
the invariant sections $\Gamma(\cH^\bullet)^\cG$ form a subring.

\begin{proposition}\label{thm:lalgcohom}
Let $(M,\pi)$ be an s-proper Poisson manifold. For any s-proper integration $(\G,\Omega)$ of $(M,\pi)$ there is a canonical ring isomorphism:
\[
 H_\pi^\bullet(M)\cong \Gamma(\cH^\bullet)^\cG.
\]
 
%
\end{proposition}

\begin{proof}
Let $(\Omega_s(\cG),\d_s)$ denote the fiberwise de Rham chain complex  of $s: \cG\to M$ and let $H(\Omega_s(\cG),\d_s)$ denote its homology.
First, recall that there is a canonical isomorphism
\begin{equation}\label{cohom-isom-mult}
(\X^{\bullet}(M),\d_\pi)\cong (\Omega_s^\bullet(\cG)^\G,\d_s)
\end{equation}
between the chain complex computing Poisson cohomology and the $\G$-invariant part of the complex $(\Omega_s(\cG),\d_s)$.

Now, there is an obvious ring homomorphism
\[
H(\Omega_s^\bullet(\cG),\d_s)\rightarrow \Gamma(\cH^\bullet),
\]
which sends the class of a fiberwise closed form to the section whose value that at $x$ is the de Rham class of the 
restriction of the the form to $s^{-1}(x)$. In fact, it is well-known that this map is an isomorphism.
Restricting to $\G$-invariant forms we obtain a ring monomorphism
\begin{equation}\label{cohom-isom-Ginv}
H(\Omega_s^\bullet(\cG)^\G,\d_s)\rightarrow \Gamma(\cH^\bullet)^\cG
 \end{equation}
We claim that this is an isomorphism. Given $c\in  \Gamma(\cH^\bullet)^\cG$ we can represent it by  $\a\in \Omega_s^\bullet(\cG)$ fiberwise closed.
In principle $\a$ need not be $\G$-invariant, but this can be fixed by a standard averaging argument: since $\cG$ is a (source-)proper
groupoid, we can choose a left-invariant Haar measure $\d\nu$, i.e., a measure on each $t$-fiber of total mass 1, which
is invariant by the left action of $\G$. Then we form the average $\int_t\a ~\d\nu$: it is the fiberwise closed form whose restriction to $s^{-1}(x)$ is given by
\[\left(\int_t\a  ~\d\nu\right)_{ s^{-1}(x)}:=\int_{t^{-1}(x)}\a_{s^{-1}(s(h))}(h_* \cdot)~ \d\nu_{t^{-1}(x)},\quad h\in t^{-1}(x).\] 
This average lies in $\Omega_s^\bullet(\cG)^\G$. Moreover, since $\d\nu$ has total mass 1 on each $t$-fiber and $c\in  \Gamma(\cH^\bullet)^\G$, the fiberwise
cohomology class of $\int_t\a ~\d\nu$ is the same as that of $c$. This proves the proposition.

\end{proof}

Let us assume now that the s-proper Poisson manifold $(M,\pi)$ is oriented. This implies that the source fibers of
any s-connected, s-proper symplectic integration $(\cG,\Omega)$ can be coherently oriented. An orientation on $s^{-1}(x)$
defines the classical Poincar\'e duality pairing on $\cH^\bullet_x$ taking values in $\cH^0_x\equiv\R$;
since the fibers are coherently oriented, then the fiberwise pairings define a pairing on 
$\Gamma(\cH^\bullet)$ with values in 
$\Gamma(\cH^0)\equiv C^\infty(M)$,
inducing a corresponding pairing on  Poisson cohomology:

\begin{theorem} Let $(M,\pi)$ be an oriented s-proper Poisson manifold. Then any s-connected, s-proper symplectic 
 integration $(\cG,\Omega)$ canonically defines a {\bf Poincar\'e duality pairing}
 \[\langle \cdot,\cdot\rangle_\cG:H^{\bullet}_\pi(M)\times H^{\mathrm{top}-\bullet}_\pi(M) \longrightarrow C^\infty(M/\cF)\]
 which is non-degenerate.
\end{theorem}
\begin{proof}
For oriented manifolds the Poincar\'e duality pairing is preserved by isomorphisms in cohomology induced by orientation preserving
diffeomorphisms. The right action by an arrow $g\in \cG$ preserves the orientation of the s-fibers because units clearly have this property and $\cG$ is connected. Hence, the pairing in $\Gamma(\cH^\bullet)$ restricts to a pairing in $\Gamma(\cH^\bullet)^\cG$. This pairing on $\Gamma(\cH^\bullet)^\cG$  is transferred via the isomorphism 
\eqref{cohom-isom-mult} to a  pairing on $H^\bullet_\pi(M)$ with values in $C^\infty(M/\cF)\equiv \Gamma(\cH^0)^\cG$.

It remains to check that the pairing defined in $\Gamma(\cH^\bullet)^\cG$ is non-degenerate. Let $c\in \Gamma(\cH^k)^\cG$ be
such that $c(x)\neq 0$. Then we can always find $b\in \Gamma(\cH^{\mathrm{top}-k})$ such that at $x$ their classical Poincar\'e duality pairing satisfies
\[\langle c(x),b(x)\rangle=1.\]
We will produce $\widetilde{b}\in \Gamma(\cH^{\mathrm{top}-k})^\G$ such that:
\[\langle c(x),\widetilde{b}(x)\rangle=1,\]
which shows that  the pairing defined in $\Gamma(\cH^\bullet)^\cG$ is non-degenerate. 

In the proof of Proposition \ref{thm:lalgcohom} we saw that the sections $c$ and $b$ can be represented by fiberwise closed forms $\a$ and $\beta$,
where $\a$ is $\cG$-invariant. We use again averaging to produce $\int_t\beta~\d\nu\in \Omega_s(\cG)^\G$, whose image by (\ref{cohom-isom-Ginv})
we denote by $\widetilde{b} \in \Gamma(\cH^k)^\cG$. We claim that $\langle c(x),\widetilde{b}(x)\rangle=1$. Indeed,
\begin{align*}
\langle c(x),\widetilde{b}(x)\rangle&=\langle \a_{s^{-1}(x)},\left(\int_t\beta ~\d\nu\right)_{s^{-1}(x)}\rangle\\
&=\langle\left(\int_t\a ~\d\nu\right)_{s^{-1}(x)},\left(\int_t\beta ~\d\nu\right)_{s^{-1}(x)}\rangle,
\end{align*}
and since averaging commutes with the classical Poincar\'e duality pairing on any s-fiber,  we obtain
\[\langle c(x),\widetilde{b}(x)\rangle= \int_t \langle \alpha_{s^{-1}(x)},\beta_{s^{-1}(x)}\rangle ~\d\nu=
\int_t \langle c(x),b(x)\rangle ~\d\nu=1.\]
 \end{proof}

Clearly, the previous results which were stated for Poisson manifolds, are really about the Lie algebroid 
cohomology $H^\bullet(A)$ of s-proper Lie algebroids: if $\cG$ is an s-proper integration of  $A\rightarrow M$, 
then there is a ring isomorphism
\[H^\bullet(A)\cong \Gamma(\cH^\bullet)^\cG.\]
Likewise, if $A$ is orientable, so the (connected) $s$-fibers of $\cG$ can be coherently oriented, then 
$\cG$ defines a non-degenerate Poincar\'e duality pairing on $H^\bullet(A)$ with values in $C^\infty(M/\G)$.

\begin{remark} The Poincar\'e duality pairing for s-proper Lie groupoids takes values in Casimirs. 
A Hamiltonian invariant volume form on a compact Poisson manifold $(M,\pi)$ determines a real valued pairing:
\[H^\bullet_\pi(M)\times H^{\mathrm{top}-\bullet}_\pi(M) \to \R,\quad (P,Q)\mapsto \int_M (i_{P\wedge Q}\mu^{-1})~\mu.\]
In general, it is not known whether these pairings are non-degenerate (see \cite{ELW} for details). In \cite{CFMc} we 
will show that on Poisson manifolds of compact type the real valued pairings defined by invariant density measures are non-degenerate, and we shall
describe their relation with the Casimir-valued Poincar\'e duality pairing(s) above.
\end{remark}

\subsection{Hodge decomposition}
We now describe a Hodge-type decomposition for the complex $(\X^\bullet(M),\d_\pi)$ of any $s$-proper Poisson manifold. 

Let us first recall the standard Hodge decomposition: given an oriented Riemannian manifold $N$ one has the associated Laplace-Beltrami operator:
\[ \Delta= \d\delta + \delta \d: \Omega^{\bullet}(N)\to \Omega^{\bullet}(N), \]
where $\delta= *\d*: \Omega^{\bullet}(N)\to \Omega^{\bullet -1}(N)$ and $*: \Omega^{\bullet}(N)\to \Omega^{\top-\bullet}(N)$ is the 
Hodge-$\ast$ operator. When $N$ is compact one obtains a direct sum decomposition
\[ \Omega^{\bullet}(N)= \Ker{\Delta}\oplus \Im(\d)\oplus \Im(\delta) .\]
The resulting projection onto $\Ker{\Delta}$ is usually denoted by $\cH$. One often uses the Green operator $G$,
which gives the unique solution of $\Delta\omega=\alpha$ in
$(\Ker\Delta)^\perp$, to write
\[ \alpha= \cH(\alpha)+ \d \delta G(\alpha)+ \delta \d G(\alpha),\quad \forall\ \alpha\in \Omega^{\bullet}(N).\]

We extend this decomposition to any Lie algebroid $A\to M$ as follows. First, a choice of metric on the vector bundle $A$,
giving rise to the $A$-Laplacian:
\[ \Delta_{A}= \d_A\delta_A+ \delta_A\d_A: \Omega^{\bullet}(A)\to \Omega^{\bullet}(A).\]
As a straightforward generalization of the standard Hodge decomposition, one obtains a similar 
decomposition for $\Omega^{\bullet}(A)$ when $M$ is compact and the $A$-de Rham complex is elliptic. This is too restrictive
since ellipticity holds only when $A$ is transitive. However, there is yet another generalization, which holds under 
the compactness conditions that are of interest to us. For simplicity, we shall state it just for the cotangent Lie algebroid
of a Poisson manifold:

\begin{theorem} Let $(\cG,\Omega)$ be an s-connected, s-proper symplectic integration of an oriented Poisson manifold $(M,\pi)$. Then, upon a choice of metric on $M$, one has:
\begin{enumerate}[(i)]
\item a direct sum decomposition:
\[ \X^{\bullet}(M)= \Ker{\Delta_{T^*M}}\oplus \Im(\d_{T^*M})\oplus \Im(\delta_{T^*M});\]
\item an isomorphism between Harmonic multivector fields and Poisson cohomology:
\[\cH(\X^{\bullet}(M))\cong H^\bullet_\pi(M);\]
\item Hodge-$\ast$ isomorphisms:
 \[H^{\bullet}_\pi(M) \cong H_\pi^{\mathrm{top}-\bullet}(M).\]
\end{enumerate}
\end{theorem}

\begin{proof} Using right translations, the metric on $M$ induces Riemannian metrics on all the s-fibers of $\cG$ 
and we can apply the Hodge decomposition to all the s-fibers: on the fiberwise de Rham complex $(\Omega_{s}^{\bullet}(\cG),\d_s)$
we obtain the additional operators $\delta_s$ and $\Delta_s$. Since the standard Hodge decomposition and Green operator
have smooth dependence on the metric, we obtain a Hodge-type decomposition 
\[    \Omega_{s}^{\bullet}(\cG)= \Ker{\Delta_s}\oplus \Im(\d_s)\oplus \Im(\delta_s).\]
Moreover, since the standard Hodge decomposition is invariant under isometries, one obtains a similar decomposition of the invariant part 
\[    \Omega_{s}^{\bullet}(\cG)^\G\cong \X^{\bullet}(M),\]
and it is straightforward to check that the restrictions of $\d_s$, $\delta_s$ and $\Delta_s$ to this invariant part become
$\d_{T^*M}=\d_{\pi}$, $\delta_{T^*M}$ and $\Delta_{T^*M}$, respectively. This proves the decomposition (i).

Just like in the manifold case, the cohomology $H(\Omega_{s}^{\bullet}(\cG)^\cG,\d_s)$ is isomorphic to the subspace of Harmonic sections
\[\cH(\Omega_{s}^{\bullet}(\cG)^\cG),\]
and this isomorphism is transferred to an isomorphism between harmonic multivector fields and Poisson cohomology, showing that (ii) holds.

Similarly, the Hodge-$\ast$ operator 
\[\cH^\bullet(\Omega_{s}(\cG)^\cG)\to \cH^{\mathrm{top}-\bullet}(\Omega_{s}(\cG)^\cG)\]
is an isomorphism, since it squares to $\pm\mathrm{Id}$.
\end{proof}



\begin{example} 
Let $M=\mathfrak{g}^*$ where $\gg$ is a compact Lie algebra. It is proved in \cite{GiWe} that:
\[H_{\pi}^\bullet(\gg^*)\cong H^\bullet(\gg)\otimes \text{Cas}(\gg^*).\]
We can interpret this result in terms of the Hodge decomposition as follows. 

Fix a compact Lie group $G$ integrating
$\mathfrak{g}$ and an inner product on $\mathfrak{g}^*$. Then the harmonic multivector fields in $\gg^*$ can be described as follows: 
we let $\cH_{CE}$ be the vector space of harmonic representatives of the Chevalley-Eilenberg complex of $\mathfrak{g}$
with respect to the given inner product. They can be viewed as constant multivector fields in $\gg^*$. 
Note that these are the harmonic representatives of the cohomology of $G$ with respect to the right invariant metric 
defined by the inner product.
 
Now let $h_1,\dots, h_k$ be a basis of $\cH_{CE}$, which we regard as sections of the trivial bundle over $\mathfrak{g}^*$ with fiber $\cH_{CE}$. 
A harmonic multivector field $P$ is a $G$-invariant section of this trivial bundle. Since the $h_i$ are already $G$-invariant, we must have 
 \[P=\sum_{i=1}^k f_ih_i,\]
where the $f_i\in\text{Cas}(\gg^*)$.
\end{example}

We close this section stating a few consequences of the Hodge decomposition. First,
from Propositions \ref{prop:first:cohomology} and \ref{prop:h2:non-vanishing}, we deduce:

\begin{corollary} Let $(M,\pi)$ be an orientable Poisson manifold. Then: 
 \begin{enumerate}[(i)]
  \item if $(M,\pi)$ is s-proper, then $H^{\mathrm{top}-1}_\pi(M)=0$;
  \item if $(M,\pi)$ is of strong compact type, then $H^{\mathrm{top}-2}_\pi(M)\neq 0$.
 \end{enumerate}
\end{corollary}

For a second consequence, we need to recall Poisson homology (see (\cite{Br}). In this 
homology theory the chains are the differential forms and the boundary operator $\delta_\pi$ is given by
\[ \delta_\pi=i_{\pi}\d-\d i_{\pi}.\]
It was observed in \cite{ELW} that a Hamiltonian invariant volume form $\mu$ establishes an isomorphism of complexes
\[ (\X^{\bullet}(M),\d_\pi)\cong (\Omega^{\top-\bullet}(M),\delta_\pi), \quad P\mapsto i_P\mu_0,\] 
and so induces an isomorphism between Poisson cohomology and homology spaces $H^\bullet_\pi(M)\cong H_{\top-\bullet}^\pi(M)$. 
Combining this isomorphism with the Hodge-$\ast$ isomorphism 
we conclude that:

\begin{corollary}
If $(M,\pi)$ in an s-proper orientable Poisson manifold, then Poisson cohomology and homology coincide:
\[ H^\bullet_\pi(M)\cong H_\bullet^\pi(M). \]
In particular, we have:
\[ C^\infty(M/\cF)\cong C^\infty(M)/\{C^\infty(M),C^\infty(M)\}.\]
\end{corollary}
\begin{proof}
 The right hand side of the second equality is, by definition, $H_0^\pi(M)$.
\end{proof}

\section{Normal linear form around leaves}
\label{sec:Moser}

In symplectic geometry, Moser-Weinstein type techniques provide neighborhood theorems and normal forms for symplectic manifolds 
around submanifolds of special type, such as Lagrangian, isotropic, coisotropic or symplectic submanifolds. Such theorems cannot hold 
for general Poisson manifolds, because  normal forms depend on the first order jets of the structures (symplectic/Poisson) along the submanifold.
Even at fixed points, Poisson structures may have trivial first jet but non-trivial higher order jets (e.g., quadratic Poisson structures).
We will show now that for PMCTs tubular neighborhoods and normal forms around leaves do exist.

\subsection{Neighborhoods of symplectic leaves}

A symplectic manifold can be a symplectic leaf of distinct Poisson structures, even of different dimensions.
On the other hand, if $(M,\pi)$ is a Poisson manifold and $S$ a symplectic leaf with symplectic form $\omega_S$, then we have the restricted Lie
algebroid $A_S:=T^*_SM$, which is a transitive Lie algebroid over the leaf. One then may ask the following fundamental questions:
\begin{itemize}
\item {\bf Realization Problem}: Given a transitive Lie algebroid $A_S\to S$ over a symplectic manifold $(S,\omega_S)$, is there a Poisson
manifold $(M,\pi)$ such that $(S,\omega_S)$ is a symplectic leaf of $M$ and $A_S\cong T^*_SM$? If this is the case, then we say that $A_S\to (S,\omega_S)$ is {\bf realizable in the Poisson manifold} $(M,\pi)$;
\item {\bf Neighborhood Equivalence}: Given a transitive algebroid $A_S\to S$ over a symplectic manifold $(S,\omega_S)$ 
which is realizable in two Poisson manifolds $(M_1,\pi_1)$ and $(M_2,\pi_2)$, is there a Poisson diffeomorphism $\phi:U_1\to U_2$
defined in some neighborhoods $S\subset U_1\subset M_1$ and $S\subset U_2\subset M_2$, which restricts to a symplectomorphisms in $S$? 
If this is the case, then we say that $(M_1,\pi_1)$ and $(M_2,\pi_2)$ are {\bf neighborhood equivalent} around $A_S$;
\item {\bf Local Rigidity:} Given a transitive Lie algebroid $A_S\to S$ over a symplectic manifold $(S,\omega_S)$ which is realizable, 
are any two realizations of $A_S$ neighborhood equivalent? In this case we say that $A_S\to (S,\omega_S)$ is {\bf locally rigid}.
\end{itemize}
 
The local linear model discussed in Section \ref{sec:local:model} gives a simple solution to the realization problem whenever $A_S$ is integrable:
 
 \begin{proposition}
 \label{prop:realization:problem}
If $A_S\to (S,\omega_S)$ is an integrable transitive Lie algebroid over a symplectic manifold, then it is realizable in a Poisson manifold $(M,\pi)$. Moreover:
\begin{enumerate}[(i)]
\item if $A_S$ has an s-connected, proper (respectively, source 1-connected, proper) integration, then $(M,\pi)$ can be taken to be of proper
(respectively, strong proper) type;
\item if $A_S$ has an s-connected, s-proper (respectively, source 1-connected, s-proper) integration, then $(M,\pi)$ can be taken
to be of s-proper (respectively, strong s-proper) type.
\end{enumerate}
 \end{proposition}
 
 \begin{proof}
 Let $\G_S\tto S$ be an s-connected integration of $A_S\to S$. Since $\G_S$ is transitive, it is isomorphic to a gauge groupoid $P\times_{G_x} P\tto S$,
 where we can take $P$ to be the principal $G_x$-bundle $t: s^{-1}(x)\to S$, for any fixed $x\in S$. Now choose a principal bundle connection 
 $\theta:TP\to\gg$ in $P$ and build the local linear model $(M,\pi^\theta_{\lin})$. One checks easily that $(S,\omega_S)$ is a
 symplectic leaf of $(M,\pi^\theta_\lin)$ and that $A_S\cong T^*_SM$. Items (i) and (ii) follow from Corollary \ref{cor:local:model}.
 \end{proof}
 
\subsection{Neighborhood equivalence}
We now consider the neighborhood equivalence problem. The proof of the Proposition \ref{prop:realization:problem} shows that the
Poisson manifold $(M,\pi)$ where we realize $A_S\to (S,\omega_S)$ depends crucially on the choice of integration. Namely, we have
$S\subset M\subset P\times_G\gg^*$, where $P$ is the principal $G$-bundle defined by the choice of integration
$\G_S\tto S$ (cf.~Section \ref{sec:local:model}). Hence, we will fix the integration $\G_S\tto S$ and ask how unique
are realizations $(M,\pi)$ with symplectic integration $\G\tto M$ such that $\G|_S=\G_S$.

Let us start by remarking that any integration $\G_S\tto S$ of $A_S\to (S,\omega_S)$ comes with a multiplicative closed 2-form, namely:
\[ \Omega_S:=t^*\omega_S-s^*\omega_S. \]
It is easy to check that $\Omega_S$ has constant rank equal to $2\dim S$, and that the target map is a forward Dirac map: $t_*L_{\Omega_S}=L_{\omega_S}$.

We can then reformulate the neighborhood equivalence question as the following problem at the groupoid level: given two embeddings 
\[ j_i: (\G_{S}, \Omega_{S}) \hookrightarrow (\G_i, \Omega_i),\quad (i=0,1),\]
where $j_i(S)$ is an orbit of $\G_i$, is there a symplectic groupoid neighborhood equivalence between $(\G_0,\Omega_0)$ and $(\G_1,\Omega_1)$?

\begin{theorem}\label{Mo-We-type-thm} Let $(\G_{S},\Omega_S)\tto (S,\omega_S)$ be an integration of $A_S\to (S,\omega_S)$. Assume that we have two embeddings 
\[ j_{i}: (\G_{S}, \Omega_{S}) \hookrightarrow (\G_{i}, \Omega_{i}), \quad (i=0,1)\]
into proper symplectic groupoids. Then there exist neighborhoods $U_i$ of $S$ in $M_i$ and an isomorphism of symplectic groupoids 
\[ \Phi: (\G_{0}|_{U_0}, \Omega_0) \overset{\cong}{\longrightarrow} (\G_{1}|_{U_1}, \Omega_1)\]
which is the identity on $\G_{S}$ (i.e., $\Phi\circ j_0= j_1$).
\end{theorem}

Before we turn to the proof, we would like to point out some important consequences: 

\begin{corollary}[Local Linear Normal Form]
\label{cor:local:normal:form}
Let $(M,\pi)$ be a proper Poisson manifold and $S$ a symplectic leaf through $x\in M$. If $\G\tto M$ is any s-connected, proper symplectic integration, then
there are neighborhoods $S\subset U \subset M$ and $S\subset V\subset s^{-1}(x)\times_{G_x}\gg^*_x$ such that $\pi|_U$ is Poisson diffeomorphic to the local linear model
${\pi^\theta_\lin}|_V$.
\end{corollary}

\begin{proof}
Consider the local linear model $\pi^\theta_\lin$ associated to the principal $G_x$-bundle $t:s^{-1}(x)\to S$, for some choice of connection $\theta$. 
It is not hard to check that it
admits the symplectic integration:
\[ (s^{-1}(x)\times s^{-1}(x)\times \gg^*)/G_x\tto (s^{-1}(x)\times \gg^*)/G_x, \]
where the multiplicative symplectic form is induced by the basic form:
\[ \Omega^{\theta}_{\lin}= (t\circ\mathrm{pr}_1)^*\omega-(t\circ\mathrm{pr}_2)^*\omega-\d\langle \mathrm{pr}_1^*\theta,\cdot\rangle+\d\langle \mathrm{pr}_2^*\theta,\cdot\rangle\in \Omega^2(s^{-1}(x)\times s^{-1}(x)\times \gg^*). \]
To be more precise, one has to replace $\gg_x^*$ by some $G_x$-invariant neighborhood of the origin for this form to be symplectic.

Now, if in Theorem \ref{Mo-We-type-thm} we let $(\G_1,\Omega_1)=(\G,\Omega)$ and we let $(\G_0,\Omega_0)$ be this symplectic integration 
of the local linear model $\pi^\theta_\lin$, then the resulting groupoid neighborhood equivalence covers a Poisson 
diffeomorphism between $\pi|_U$ and ${\pi^\theta_\lin}|_V$, where $U$ is a neighborhood of $S$ in $M$ and $V$ is a neighborhood of $S$ in the local linear model.
\end{proof}

\begin{remark}
In general, the neighborhoods $U_0$ and $U_1$ in Theorem \ref{Mo-We-type-thm} (or $U$ and $V$ in Corollary \ref{cor:local:normal:form})
will not be saturated by orbits (i.e., symplectic leaves). However, in the s-proper case we can take them to be saturated, and so symplectic
leaves of s-proper Poisson manifolds have open, saturated neighborhoods where the Poisson structure looks like the local linear model. In \cite{CFMc},
we shall prove that even for Poisson manifolds of proper type there exist normal forms on saturated neighborhoods, but they will no longer be linear!
\end{remark}

Another important consequence of Theorem \ref{Mo-We-type-thm} is the following rigidity result:

\begin{corollary}[Local Rigidity]
\label{cor:local:rigidity}
Let $A_S\to (S,\omega_S)$ be a transitive Lie algebroid over a symplectic manifold which integrates to a compact, source 1-connected Lie groupoid.
Then $A_S$ is locally rigid inside integrable Poisson manifolds.
\end{corollary}

\begin{proof}
If we realize $A_S$ into an integrable Poisson manifold $(M,\pi)$, then, by Ehresmann's theorem, the source 1-connected 
integration of $(M,\pi)$ is s-proper. Hence, by the previous corollary, any two such realizations are neighborhood equivalent.
\end{proof}

\subsection{Moser-Weinstein techniques for proper groupoids}

We now turn to the proof of the groupoid neighborhood equivalence (Theorem \ref{Mo-We-type-thm}). This will proceed in two steps: first,
we use a groupoid tubular neighborhood theorem which linearizes the groupoid. In the second step, we apply
a groupoid version of the classic Moser-Weinstein path method. In both steps, properness is a crucial hypothesis.

For a general Lie groupoid $\G\tto M$ and an orbit $S\subset \G$,
the normal bundle of $\G_{S}$ in $\G$ is not only a vector bundle over $\G_{S}$ but also a Lie groupoid over the normal bundle of $S$ in $M$:
\[ \xymatrix{ \nu(\G_S) \ar@<0.25pc>[r] \ar@<-0.25pc>[r]  \ar[d] & \nu(S) \ar[d] \\
\G_S \ar@<0.25pc>[r] \ar@<-0.25pc>[r] & S} \]
The structural maps of $\nu(\G_{S})$ are the ones induced by the differentials of the structural maps of $\G$.
The normal bundle $\nu(\G_S)\to \G_S$ is an example of a {\bf vector bundle groupoid} (VB groupoid, for short).

One has a tubular neighborhood theorem, known as the Zung-Weinstein linearization theorem \cite{We3, Zu}, which 
can be stated as follows (see also \cite{CS, HoFe} for more geometric proofs and further details):

\begin{theorem}
\label{thm:linearization:groupoids}
If $S$ is an orbit of a proper Lie groupoid $\G\tto M$, then there exist a neighborhood $U$ of $S$ in $M$, a neighborhood $V$ of $S$ in 
$\nu(S)$, and an isomorphism of Lie groupoids
\[ \G|_U \cong  \nu(\G_S)|_V,\]
which is the identity on $\G_S$. Furthermore, if $\G$ is $s$-proper, then one can choose $U$ to be saturated and
$V= \nu(S)$. 
\end{theorem}

In \cite{HoFe} the linearization map of Theorem \ref{thm:linearization:groupoids} is obtained as the exponential map 
\[ \exp_\eta: \nu(\G_S)|_V\to \G|_U \]
of a special kind of metric $\eta$, known as a 2-metric.  When $(\G,\Omega)$ is a symplectic groupoid,
we saw that the linear model is also a symplectic groupoid $(\nu(\G_S),\Omega^\theta_\lin)$,
where the symplectic form depends on a choice of a principal bundle connection $\theta$.
We show in the Appendix that in this case we have the following symplectic version of Theorem \ref{thm:linearization:groupoids}:

\begin{proposition}
\label{prop:linearization:sympl:groupoids}
Let $S$ be an orbit of a proper symplectic groupoid $(\G,\Omega)\tto M$. Then one can choose a 2-metric $\eta$ on $\G$, a principal  connection $\theta$ on $t:s^{-1}(x)\to S$, and a groupoid automorphism  $\Phi:\nu(\G_S)\to \nu(\G_S)$, such that the composition:
\[ \exp_\eta\circ\,\Phi: \nu(\G_S)|_V\to \G|_U, \]
is a groupoid isomorphism and satisfies:
\[ ((\exp_\eta\circ\,\Phi)^*\Omega)|_g=\Omega^\theta_\lin|_g,\quad \forall g\in \G_S. \]
\end{proposition}

%

Returning to the proof of Theorem \ref{Mo-We-type-thm}, after applying the symplectic version of the linearization theorem, we
may assume that $\G_0= \G_1$ is a VB groupoid $\cH$ over the groupoid $\cH_{S}$, so we find ourselves in the following setting:
\begin{itemize}
\item[(a)] $\cH\tto M$ is a proper Lie groupoid and $S$ is an orbit;
\item[(b)] $\Omega_0$ and $\Omega_1$ are multiplicative symplectic forms on $\cH$ such that: 
\[ {\Omega_0}|_h= {\Omega_1}|_h,\quad  \forall h\in \cH_{S};\]
\item[(c)] There is a smooth family of groupoid maps $\{h_{t}: \cH\to \cH\}_{t\in [0, 1]}$ --given by the flow of minus the Euler vector field--
such that $h_{1}= \textrm{Id}$, $h_{t}$ is the identity on $\cH_S$, and $h_0(\cH)= \cH_{S}$.
\end{itemize}

The next step is the groupoid version of the usual Moser-Weinstein path method:

\begin{lemma} Assume that we are in the situation (a)-(c) above. Then the conclusion of Theorem \ref{Mo-We-type-thm} holds: there exist open neighborhoods $U_0$ and $U_1$ of $S$ in $M$ and
a diffeomorphism of symplectic groupoids
$(\cH|_{U_0}, \Omega_0) \cong (\cH|_{U_1}, \Omega_1)$
which is the identity on $\cH_{S}$.
\end{lemma}

\begin{proof} We consider the groupoid version of the standard Moser lemma. The main steps of the standard argument are:
\begin{enumerate}[(i)]
\item Choosing a 1-form $\alpha$ such that $\Omega_0- \Omega_{1}= \d\alpha$ and $\alpha|_{\cH_S}=0$
(possible in a neighborhood of $\cH_S$ by (b)).
\item Considering the family of 2-forms:
\[ \Omega_t:=t\Omega_1+(1-t)\Omega_0, \quad t\in [0, 1],\]
(symplectic in a neighborhood of $\cH_S$ by (b)).
\item Considering the time-dependent vector field $X_t$ defined by the equations: 
\[ i_{X_t}\Omega_t=\alpha .\]
\end{enumerate}
The vector field $X_t$ vanishes in $\cH_S$. Hence, the flow $\phi_t$ of $X_t$ is defined up to time 1 in some neighborhood $\widetilde{U}$ of $\cH_S$,
it is the identity on $\cH_S$, and $\Omega_0= \phi_{1}^{*}\Omega_1$ in $\widetilde{U}$. Assuming for the moment that $\alpha$ can be arranged 
to be multiplicative, we can ensure that:
\begin{itemize}
\item The vector fields $X_t$ are multiplicative (since $\Omega_{t}$ and $\alpha$ are multiplicative).
In particular, $X_t$ is s-projectable and t-projectable to the same vector field $V_t$ on $M$.
\item $\widetilde{U}= \cH_{U}$ for some neighborhood $U$ of $S$ in $M$. This follows from two facts: that
a multiplicative form on a groupoid is everywhere non-degenerate iff it is non-degenerate at units (e.g.,
see \cite[Lemma 4.2]{BCWZ}) and that, since $X_t$ is multiplicative and $\cH$ is proper,
the flow of $X_t$ is defined for as long as the flow of $V_t$ is defined. Hence one just chooses $U$ accordingly 
(which is possible because $V_t$ is zero on $S$).
\item Since $X_t$ is multiplicative, its flow $\phi_t$ is by groupoid automorphisms.
\end{itemize}

Therefore, we are left with proving the existence of a multiplicative 1-form $\alpha$ as above, which is the content of the next lemma.
Note that it is there (and only there) that condition (c) is needed. 
\end{proof}

\begin{lemma} Assume that (a)-(c) hold (but properness is not needed here).
Then there exists a multiplicative 1-form $\alpha$ on $\cH$ such that:
\[ \Omega_0- \Omega_{1}= \d\alpha,\quad \alpha|_{\cH_S}= 0.\]
\end{lemma} 

\begin{proof} Again, we have to make sure that the standard argument (basically the Poincar\'e Lemma)
preserves multiplicativity.  Since $\Omega:= \Omega_0- \Omega_1$ is closed and zero on $\cH_S\subset \cH$,
the standard argument gives us an explicit $\alpha\in \Omega^1(\cH)$ satisfying the equations in the statement, namely:
\[ \alpha(X_g)= \int_{0}^{1} \Omega( \frac{\d}{\d t} h_t(g), \d_g h_t(X_g) ) \d t .\]
Moreover, since each $h_t$ is a groupoid homomorphisms it follows that
\[ (\d_{gh} h_t)\circ(\d_{(g,h)} m) (X_g, X_h)= (\d_{(h_t(g),h_t(h))} m) (\d_g h_t(X_g), \d_h h_t(X_h)),\]
and
\[ \frac{\d}{\d t} h_t(gh)= \d_{(h_t(g),h_t(h))}m(\frac{\d}{\d t} h_t(g), \frac{\d}{\d t} h_t(h)),\]
where $m$ is the multiplication map of $\cH$ and $(X_g, X_h)$ is any vector tangent to the domain of $m$. Using 
also the multiplicativity of $\Omega$, we immediately deduce that $\alpha$ is, indeed, multiplicative.
\end{proof}

\appendix
\section{Linearization of proper symplectic groupoids}

In this appendix we give a proof  of the following symplectic version of the
Zung-Weinstein linearization theorem, which is needed for the proof of the symplectic
groupoid neighborhood equivalence theorem (Theorem \ref{Mo-We-type-thm}):

\begin{proposition}
\label{prop:symplectic:linearz}
Let $S$ be an orbit of a proper symplectic groupoid $(\G,\Omega)\tto M$. 
Then one can choose a 2-metric $\eta$ on $\G$, a principal  connection $\theta$ on $t:s^{-1}(x)\to S$,
and a groupoid automorphism  $\Phi:\nu(\G_S)\to \nu(\G_S)$, such that the composition:
\[ \exp_\eta\circ\,\Phi: \nu(\G_S)|_V\to \G|_U, \]
is a groupoid isomorphism and satisfies:
\[ ((\exp_\eta\circ\,\Phi)^*\Omega)|_g=\Omega^\theta_\lin|_g,\quad \forall g\in \G_S. \]
\end{proposition}

We will not go into details on 2-metrics on a groupoid $\G\tto M$. For our purpose, it suffices to know that it amounts to a metric $\eta^{(0)}$ on the objects $M$,
a metric $\eta^{(1)}$ on the arrows $\G$ and a metric $\eta^{(2)}$ on the composable arrows $\G^{(2)}$ 
which are compatible with all structure maps. The details can be found in \cite{HoFe}. 
We will focus our attention on the metric on the arrows $\eta^{(1)}$, which we write simply as $\eta$.
As we mentioned above, the main property of a 2-metric is that the exponential map $\exp:\nu(G_S)\to \G$ is a groupoid morphism.

Recall that a {\bf multiplicative distribution} in a groupoid $\G\tto M$ is, by definition, a subgroupoid $\cD\tto D$ of the tangent groupoid $T\G\tto TM$. 
We start by remarking that for a proper groupoid, if one is given a multiplicative distribution $\cD\tto D$ in $T_{\G_S}\G\tto T_SM$ complementary to $T\G_S\tto TS$, 
then we can choose a 2-metric which is adapted to the decomposition $T_{\G_S} \G=T\G_{S}\oplus \cD$.

\begin{lemma}
\label{lemma:A:1}
If $S$ is an orbit of a proper Lie groupoid $(\G,\Omega)\tto M$ 
and $\cD\tto D$ in $T_{\G_S}\G\tto T_SM$ is a multiplicative distribution  complementary to $T\G_S\tto TS$, then we can choose a 2-metric in $\G$ such that:
\[ \cD=(T\G_S)^\perp,\quad D=(TS)^\perp. \]
\end{lemma}

\begin{proof}
The existence of a 2-metric on a proper groupoid is proved in \cite{HoFe} by an averaging procedure. 
Let us indicate how choices should be made so that the average procedure yields a 2-metric with the desired property.

First, one needs to fix an Ehresmann connection for the source fibration, i.e.,  to choose a splitting of the short exact sequence:
\[
\xymatrix{\ 0\ \ar[r] &\ t^* A \ \ar[r] &\  T\G \ \ar[r]^{s_*} & \ s^*TM \ \ar@/^/[l]^{\sigma} \ar[r]& \ 0\ }.
\]
Such a choice allows one to lift any groupoid action $\theta:\G\action E$ to a quasi action $T_\sigma\theta:\G\action TE$. 
We choose a splitting $\sigma:s^*TM\to T\G$ as follows: first, we restrict our attention to the subgroupoid $\G_S$, and choose and Ehresmann connection $H$
for the source map. Then we observe that $H\oplus \cD$ yields and Ehresmann connection for the source map over the points of $S$.
Such connection over
$S$ 
always extends to a global one. Its corresponding splitting preserves decompositions along $\G_S$:
\[ \quad \sigma(s^*TS)\subset T\G_S,\quad \sigma(s^*D)\subset \cD. \]

Second, we need to choose a metric $\eta$ on $\G$ for which the source map $s:\G\to M$ becomes a Riemannian submersion. We choose $\eta$ so that along $S$ we have additionally:
\[ \cD=(T\G_S)^\perp,\quad D=(TS)^\perp. \]
Again, this is possible because $\d s(T\G_S)=TS$ and $\d s(\cD)=D$. 

Then, we proceed as in \cite{HoFe}. On the fiber product:
\[ \G^{[n]}:=\{(g_1,\dots,g_{n+1}): s(g_1)=\cdots=s(g_n)\}, \]
we consider the restriction of the product metric $\eta\oplus\cdots\oplus\eta$. The groupoid $\G$ acts (on the right) on $\G^{[n]}$ on a proper
and free fashion and the quotient is $\G^{(n)}$. If we average the product metric on $\G^{[n]}$, using the lifted tangent action, 
we obtain a metric that descends to the quotient $\G^{(n)}$. The resulting metrics on $M=\G^{(0)}$, $\G=\G^{(1)}$ and $\G^{(2)}$ give a 2-metric.
The important remark is that because $\cD$  and $TG_S$ are multiplicative distributions, averaging does not alter 
the orthogonal decomposition:
\[ \cD=(T\G_S)^\perp,\quad D=(TS)^\perp. \]
\end{proof}

Let $(\G,\Omega)\tto M$ be a proper symplectic Lie groupoid with orbit $S$.  Note that $\G_S$ is a coisotropic submanifold and that the restriction of the symplectic form to $\G_S$ has kernel:
\[ K=\Ker(\Omega|_{\G_S})=\ker(\d s|_{T\G_S})\cap\ker(\d t|_{T\G_S}). \]
This is a multiplicative distribution. We will need a version of the standard coisotropic neighborhood theorem which is valid in our multiplicative setting. 

Let us start by recalling first the symplectic linear algebra statement, which we adapt here to the case of symplectic vector bundles:

\begin{lemma}
\label{lem:A:2}
Let $(V\to Q,\Omega)$  be a symplectic vector bundle and assume that there are sub-bundles $C,E,L\subset V$ such that:   
\begin{enumerate}[(i)]
\item $C\to Q$ a coisotropic sub-bundle;
\item $C=K\oplus E$, where $K=\Ker(\Omega|_C)$;
\item $E^{\perp_{\Omega}}=K\oplus L$, where  $L\subset V$ is a Lagrangian sub-bundle.
 \end{enumerate}
Then there is a canonical isomorphism of symplectic vector bundles 
\[ A: (V,\Omega)\cong (E\oplus K\oplus K^*, \Omega|_E+\Omega_\can),\] 
where $\Omega_\can$ denotes the standard symplectic form on $K\oplus K^*$.
\end{lemma}

Note that $\Omega|_E$ is symplectic. Hence, one has the direct sum decomposition
\[ V= C\oplus L= E\oplus K \oplus L .\]
The bundle isomorphism $A$ is obtained by combining this direct sum decomposition with the isomorphism 
$I_{\Omega}:L\to K^*$, $v\mapsto (i_v\Omega)|_{K}$.  In other words, it is the unique isomorphism that makes the following diagram commute:
\[ \xymatrix{
C \ar[r] \ar[d]^{\textrm{id}} & V \ar[r] \ar[d]^{A} & L \ar[d]^{I_\Omega} \\
C \ar[r] & E\oplus K\oplus K^* \ar[r]  & K^*
}
\]

We now turn to multiplicative versions of these results:

\begin{lemma}
\label{lemma:A:3}
Let $(\G,\Omega)\tto M$ be a proper symplectic Lie groupoid with orbit $S$.  Then one can choose multiplicative distributions $E,L\subset T_{\G_S}\G$ such that:
\begin{enumerate}[(i)]
\item $T\G_S=K\oplus E$, where $K=\Ker(\Omega|_{G_S})$ and $E$ is a symplectic sub-bundle;
\item $E^{\perp_{\Omega}}=K\oplus L$, where  $L$ is a Lagrangian sub-bundle.
 \end{enumerate}
\end{lemma}

\begin{proof}
We can assume that the groupoid has already be linearized. Hence, we look at a multiplicative symplectic form $\Omega$ in the groupoid $\nu(G_S)\tto \nu(S)$. 

Fixing a point $x\in S$, our groupoid can also be identified with
\[ 
\nu(\G_S)\cong (s^{-1}(x)\times s^{-1}(x) \times \nu_{x}(S))/G_{x}.
\]
In order words, $\nu(\G_S)\tto \nu(S)$ is a quotient by a proper and free $G_{x}$-action by groupoid automorphisms of the pair groupoid:
\[
s^{-1}(x)\times s^{-1}(x)\times \nu_{x}(S)\tto s^{-1}(x)\times  \nu_{x}(S).
\]

One now checks easily that a multiplicative form $\Omega$ on $\nu(\G_S)\tto \nu(S)$ corresponds to a $G_{x}$-basic multiplicative closed 2-form:
\[ \widetilde{\Omega}\in\Omega^2(s^{-1}(x)\times s^{-1}(x)\times  \nu_{x}(S)),\] 
which takes the form:
\[ \widetilde{\Omega}=\mathrm{pr}_1^*\overline{\omega}-\mathrm{pr}_2^*\overline{\omega},\]
where $\overline{\omega}\in \Omega^2(s^{-1}(x)\times  \nu_{x}(S))$ is a $G_{x}$-basic form on $s^{-1}(x)\times  \nu_{x}(S)$.

In general, it is difficult to produce multiplicative distributions in a groupoid. But, in our case, it is easy to check that
any choice of a $G_{x}$-invariant distribution $\overline{D}$ in $s^{-1}(x)\times \nu_{x}(S)$ gives rise to a multiplicative $G_{x}$-invariant distribution $\widetilde{D}$ in $s^{-1}(x)\times s^{-1}(x) \times \nu_{x}(S)$, 
by setting:
\[ \widetilde{D}:=\{(v_1,v_2,w)\in T(s^{-1}(x)\times s^{-1}(x) \times \nu_{x}(S)): (v_1,w),(v_2,w)\in D\}. \]
The quotient $D:=\widetilde{D}/G$ is a multiplicative distribution on $\nu(\G_S)\tto \nu(S)$. Two instances of this are:
\begin{enumerate}[(i)]
\item The restriction of the symplectic from $\Omega$ to $\G_S$ has kernel the multiplicative distribution:
\[ K=\Ker(\Omega|_{\G_S})=\ker(\d s|_{T\G_S})\cap\ker(\d t|_{T\G_S}). \]
This multiplicative distribution corresponds to the $G_{x}$-invariant distribution in $s^{-1}(x)\cong s^{-1}(x)\times \{0\}$ defined by the tangent spaces to the $G_{x}$-orbits: $\overline{K}=\Ker{\d t}$, where $t:s^{-1}(x)\to S$.
\item For any principal bundle connection $\theta$ on $t:s^{-1}(x)\to S$, the horizontal space $\overline{E}$ of $\theta$ determines
a multiplicative distribution $E$ on $\G_S$. 
\end{enumerate}

Notice that for any choice of $\theta$ the multiplicative distribution $E$ is complementary to $K$:
\[ T{\G_S}=K\oplus E.\]
so that $E\subset T_{\G_S}\G$ is a symplectic sub-bundle. Hence, to finish the proof of the lemma, it remains to exhibit a Lagrangian multiplicative distribution $L\subset T_{\G_S}\G$ such that:
\[  E^{\perp_{\Omega}}=K\oplus L. \]
For this, we choose a $G_{x}$-invariant distribution
\[ \overline{L}\subset T_{s^{-1}(x)} (s^{-1}(x)\times \nu_{x}(S)), \]
satisfying:
\[  \overline{E}^{\perp_{\overline{\omega}}}=\overline{K}\oplus \overline{L} \]
and such that $\overline{\omega}$ vanishes along $\overline{L}$. Such a distribution can be obtained by a standard averaging argument, using the fact that $\overline{\omega}$ is  $G_{x}$-invariant and has kernel $\overline{K}$.
\end{proof}

\begin{proof}[Proof of Proposition \ref{prop:symplectic:linearz}] 

Let $(\G,\Omega)\tto M$ be a proper symplectic Lie groupoid with orbit $S$. By Lemma \ref{lemma:A:3}, we can choose a principal bundle connection $\theta$ on $t:s^{-1}(x)\to S$, so that
its horizontal space determines a multiplicative distribution $E$ on $\G_S$ complementary to $K$:
\[ T{\G_S}=K\oplus E.\]
Moreover, we can choose a Lagrangian multiplicative distribution $L\subset T_{\G_S}\G$ such that:
\[  E^{\perp_{\Omega}}=K\oplus L. \]
In particular:
\[ T_{\G_S}\G=T\G_S\oplus L. \]

By Lemma \ref{lemma:A:1}, we can choose a 2-metric $\eta$ on $\G$ such that $L=(T\G_S)^\perp$. The exponential map of this 2-metric gives a linearization map:
\[ \exp_\eta: \nu(\G_S)|_V\to \G|_U,\]
where $V$ and $U$ are neighborhoods of $\G_S$ in $\nu(\G_S)$ and $\G$, respectively. Obviously, the differential of this map at a point $g\in\G_S$ gives an isomorphism:
\[  \d_g\exp_\eta:T_g\G_S\oplus\nu_g(\G_S)\to T_g\G_S\oplus (T_g\G_S)^\perp=T_g\G_S\oplus L_g. \]

On the other hand, the connection $\theta$ determines a closed 2-form $\Omega_\lin^\theta$ on $\nu(\G_S)$ whose pullback to $\G_S$ coincides with the pullback of $\Omega$ (see Section \ref{sec:local:model}). It follows from the explicit expression of $\Omega_\lin^\theta$  that in the direct sum decomposition:
\[ T_{\G_S}\G= E\oplus K \oplus L =E\oplus K\oplus \nu(G_S), \]
$E$ is symplectic, while $K$ and $L=\nu(G_S)$ are Lagrangian sub-bundles, for both $\Omega_\lin^\theta$ and $(\exp_\eta)^*\Omega$.

Now Lemma \ref{lem:A:2} gives a vector bundle automorphism $\Phi:\nu(G_S)\to \nu(G_S)$ such that 
\[ (\Phi)^*(\exp_\eta)^*\Omega|_g=\Omega_\lin^\theta|_g,\quad \forall g\in \G_S. \]
In the notation above, $\Phi=I_{(\exp_\eta)^*\Omega}\circ (I_{\Omega_\lin^\theta})^{-1}$.
Since the forms $(\exp_\eta)^*\Omega$ and $\Omega_\lin^\theta$ are multiplicative, it follows that $\Phi:\nu(G_S)\to \nu(G_S)$ is a groupoid automorphism.

This completes the proof of Proposition \ref{prop:symplectic:linearz}. 
\end{proof}


%

\end{document}